\documentclass{article}
\usepackage{arxiv}

\usepackage{graphicx}
\usepackage{natbib}

 \bibpunct[, ]{(}{)}{,}{a}{}{,}%
\usepackage{graphicx}
\usepackage{multirow}
\usepackage{hhline}

\usepackage{mathrsfs}
\usepackage{mathtools}
\usepackage{float}
\usepackage{xspace}
\usepackage{xcolor}
\usepackage{soul}
\usepackage{comment}
\usepackage{algorithm}
\usepackage[noend]{algpseudocode}
\usepackage{xurl}
\usepackage{breakurl}

\usepackage{subfig}

\usepackage{amsmath,amsthm}
\usepackage{newtxtext,newtxmath}

\newtheorem{theorem}{Theorem}
\newtheorem{example}{Example}
\newtheorem{remark}{Remark}
\newtheorem{lemma}{Lemma}
\newtheorem{proposition}{Proposition}
\newtheorem{corollary}{Corollary}

\usepackage{mathtools} % Bonus

\usepackage{appendix}
\usepackage{color}
\definecolor{strcolor}{rgb}{0.6, 0.2, 0.6}
\definecolor{commentcolor}{rgb}{0.3125, 0.5, 0.3125}
\definecolor{keycol}{rgb}{0, 0, 1}

\newcommand{\aprioriassignments}{\Psi}

\newcommand{\survival}{\ensuremath{s}}

\usepackage{bbm}

\usepackage{tikz}
\usetikzlibrary{shapes.geometric, positioning}

\newcommand{\stations}{\mathcal{S}}
\newcommand{\station}{s}

\newcommand{\Problem}{\ensuremath{\textsc{CRP}}\xspace}  

\newcommand{\ProblemNoFairness}{\ensuremath{\overline{\textsc{CRP}}}\xspace}

\newcommand{\arrivalrate}{\lambda}

\newcommand{\network}{\mathcal{N}}
                 
\newcommand{\coaches}{\mathcal{C}}
\newcommand{\coach}{c}

\newcommand{\reservedcapacity}{\omega}

\newcommand{\veccapacity}{\boldsymbol{\kappa}}

\newcommand{\seatdemand}{\boldsymbol{q}}

\newcommand{\requests}{\ensuremath{\mathcal{R}}}
\newcommand{\request}{r}

\newcommand{\price}{p}

\newcommand{\requesttypes}{\mathcal{T}}
\newcommand{\requesttype}{t}

\newcommand{\notation}[2]{\textcolor{blue}{\ensuremath{#1\coloneqq#2}}}

\newcommand{\emphmath}[1]{\textcolor{red}{\ensuremath{#1}}}

\newcommand{\arrivalorder}{\ensuremath{\pi}}

\newcommand{\setLegs}{\mathcal{L}}

\newcommand{\setCoaches}{\mathcal{C}}

\newcommand{\origStation}{{orig}}
\newcommand{\destStation}{{dest}}
\newcommand{\leg}{\mathsf{l}}
\newcommand{\numPass}{n}

\newcommand{\profit}{\ensuremath{P}}
\newcommand{\assignments}{\ensuremath{\aleph}}

\newcommand{\numrequests}{n}

\newcommand{\timetogo}{\ensuremath{i}}
\newcommand{\timehorizon}{\ensuremath{\mathcal{I}}}

\newcommand{\vecx}{\ensuremath{\boldsymbol{x}}}

\begin{document}
%%%%%%%%%%%%%%%%

% Outcomment only when entries are known. Otherwise leave as is and
%   default values will be used.
%\setcounter{page}{1}
%\VOLUME{00}%
%\NO{0}%
%\MONTH{Xxxxx}% (month or a similar seasonal id)
%\YEAR{0000}% e.g., 2005
%\FIRSTPAGE{000}%
%\LASTPAGE{000}%
%\SHORTYEAR{00}% shortened year (two-digit)
%\ISSUE{0000} %
%\LONGFIRSTPAGE{0001} %
%\DOI{10.1287/xxxx.0000.0000}%

% Author's names for the running heads
% Sample depending on the number of authors;
% \RUNAUTHOR{Jones}
% \RUNAUTHOR{Jones and Wilson}
% \RUNAUTHOR{Jones, Miller, and Wilson}
% \RUNAUTHOR{Jones et al.} % for four or more authors
% Enter authors following the given pattern:

% \RUNAUTHOR{Cardonha and Raghunathan.} %

% \RUNTITLE{Coach Reservation for Groups Requests}

\title{Coach Reservation for Groups Requests}

\author{
	% You can write out first names or use initials - either way is acceptable, but be consistent
    Carlos Cardonha$^{1}$,
    Arvind U. Raghunathan$^{2}$ 
 \and
	% Additional lines of authors should be inserted using the \and command (not \\)
	% Institution list, in a slightly smaller font
	\small$^{1}$ Department of Operations and Information Management, School of Business, University of Connecticut\\ Storrs, CT 06028, United States.\and
	\small$^{2}$Mitsibushi Electric Research Laboratories\\ Cambridge, MA 02139, United States.
	% Identify at least one corresponding author, with contact email address
	% \small$^\ast$Corresponding author. Email: tallys@miami.edu \\
	% Joint contributions can be indicated like thisto this work.
}

% Block of authors and their affiliations starts here:
% NOTE: Authors with same affiliation, if the order of authors allows,
%   should be entered in ONE field, separated by a comma.
%   \EMAIL field can be repeated if more than one author

% \ARTICLEAUTHORS{%
% \AUTHOR{Carlos Cardonha}
% \AFF{Department of Operations and Information Management, University of Connecticut, \EMAIL{carlos.cardonha@uconn.edu}} %, \URL{}}
% \AUTHOR{Arvind U Raghunathan}
% \AFF{Mitsibushi Electric Research Laboratories \EMAIL{raghunathan@merl.com}}
% % Enter all authors
% } % end of the block

\maketitle

\begin{abstract}%
Passenger transportation is a core aspect of a railway company’s business, with ticket sales playing a central role in generating revenue. Profitable operations in this context rely heavily on the effectiveness of reject-or-assign policies for coach reservations. As in traditional revenue management, uncertainty in demand presents a significant challenge, particularly when seat availability is limited and passengers have varying itineraries. We extend traditional models from the literature by addressing both offline and online versions of the coach reservation problem for group requests, where two or more passengers must be seated in the same coach. For the offline case, in which all requests are known in advance, we propose an exact mathematical programming formulation that incorporates a first-come, first-served fairness condition, ensuring compliance with transportation regulations. We also propose algorithms for online models of the problem, in which requests are only revealed upon arrival, and the reject-or-assign decisions must be made in real-time. Our analysis for one of these models overcomes known barriers in the packing literature, yielding strong competitive ratio guarantees when group sizes are relatively small compared to coach capacity—a common scenario in practice. Using data from Shinkansen Tokyo-Shin-Osaka line, our numerical experiments demonstrate the practical effectiveness of the proposed policies. Our work provides compelling evidence supporting the adoption of fairness constraints, as revenue losses are minimal, and simple algorithms are sufficient for real-time decision-making. Moreover, our findings provide a strong support for the adoption of fairness in the railway industry and highlight the financial viability of a regulatory framework that allows railway companies to delay coach assignments if they adhere to stricter rules regarding request rejections.
\end{abstract}%

% \KEYWORDS{Coach assignment; multi-dimensional knapsack problem; online optimization; random order model; fairness; transportation.}
%Use this for final submission
%\HISTORY{This paper was first submitted on January 1, 2021 and has been with the authors for 3 months for 2 revisions.}

%%%%%%%%%%%%%%%%%%%%%%%%%%%%%%%%%%%%%%%%%%%%%%%%%%%%%%%%%%%%%%%%%%%%%%

% Samples of sectioning (and labeling) in POMS
% NOTE: (1) \section and \subsection do NOT end with a period
%       (2) \subsubsection and lower need end punctuation
%       (3) capitalization is as shown (title style).
%
%\section{Introduction.}\label{intro} %%1.
%\subsection{Duality and the Classical EOQ Problem.}\label{class-EOQ} %% 1.1.
%\subsection{Outline.}\label{outline1} %% 1.2.
%\subsubsection{Cyclic Schedules for the General Deterministic SMDP.}
%  \label{cyclic-schedules} %% 1.2.1
%\section{Problem Description.}\label{problemdescription} %% 2.

% Text of your paper here

% Paper body
\section{Introduction}

This article investigates the coach reservation problem (\Problem), which encompasses reject-or-assign strategies of group requests to coaches in a single train ride with finite and non-elastic capacity, i.e., overbooking is not allowed, and the seating capacity imposes a constraint that cannot be violated. Similar models studied in the transportation literature focus on \textit{individual requests}, where reject-or-assign decisions are independent across passengers. However, a significant fraction of the passengers using transportation services travel in \textit{groups}. According to~\citet{deplano2019offline}, 29\% of passengers traveled in groups in United Kingdom stations. Moreover, a recent study has suggested that group traveling is becoming increasingly popular among younger generations in the Americas and in Europe; namely,  whereas 71\% of the people born between 1964 and 1964 (boomers) travel alone, 58\% of those born between 1981 to 1996 (millennials) travel in groups (\cite{CWTResearch}).  

From an operational standpoint, group requests differ from individual ones because all group members must be assigned to the same coach for the entire itinerary. Railway companies have strong incentives to accommodate group requests, as doing so improves customer experience and boosts both revenue and capacity utilization. However, the reject-or-assign problem becomes more challenging, as the opportunity costs of accepting ``small'' requests (few passengers and short itineraries) and, with that, losing capacity to assign ``large'' ones arriving later is magnified in this setting (we illustrate this case in Example~\ref{ex: fairness}).  
%additional dimension makes even the definition of ``small'' and ``large'' requests unclear. 
%In contrast, \textit{group requests} involve two or more passengers traveling together, which makes the decision-making process more challenging. Namely, if the railway company accepts a group request, it must ensure all members are seated in the same coach for the entire journey.
%assigning ``small'' requests (i.e., few passengers and short itineraries) may prevent more profitable ones from being fulfilled (we illustrate this case in Example~\ref{ex: fairness}). 
%Therefore, unilateral request rejections play a central role in the revenue management of passenger services in railway operations. 
% Therefore, determining an adequate policy for~\Problem is challenging, especially because demand in real-world applications is typically uncertain.
%However, demand is often uncertain, the distinction between small and large requests can be unclear, and ``ideal'' large requests are typically rare. 

% \paragraph{Motivation:}
Our study is motivated by the coach reservation problem on Shinkansen trains (popularly known as ``bullet trains'') operated by the Japan Railways Group (JR). In our study, we use the itinerary of the Tokaido Shinkansen line, which consists of five stations (four legs): Tokyo, Shin-Yokohama, Nagoya, Kyoto, and Shin-Osaka. The line has more stops, but we focus on these five because the other stops are not serviced by all Shinkansen trains. Figure~\ref{fig:shinkansen_map} presents a graphical depiction of this line. 
\begin{figure}[h]
    \centering
    \includegraphics[width=0.8\textwidth]{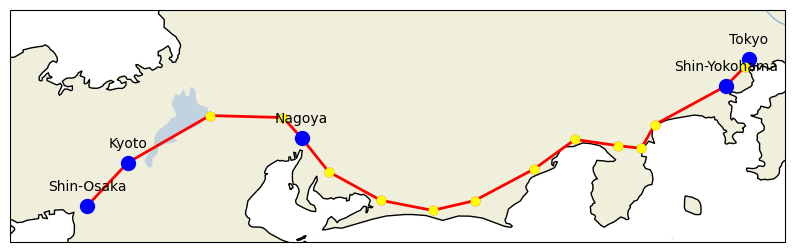}
    \caption{Map of the Tokyo-Shinosaka line (generated with Cartopy (\cite{Cartopy})). }
    \label{fig:shinkansen_map}
\end{figure}
Our main assumptions were extracted from the operational practices currently adopted by JR. 
First, we assume that all fares are fixed, 
%i.e., our model does not consider dynamic pricing (in contrast, e.g., with~\citet{manchiraju2022dynamic}). We make this assumption because 
as JR does not adopt dynamic pricing for bullet trains; some companies selling Shinkansen tickets offer early-bird discounts for a limited number of seats, but apart from that, fares are fixed. Moreover, our theoretical results and analysis are tailored and optimized for parameter values observed in real-world settings. For reference, a bullet train typically comprises 16 coaches, each with approximately 100 seats on average. We also assume that the largest group consists of six passengers, as requests with larger groups are rare and handled manually by railway companies. This aspect is important in our work because we explore the relationship between the size of the largest group and the seating capacity per coach to circumvent well-known technical barriers and improve theoretical results presented in the packing literature.

We investigate different versions of~\Problem, defined according to the arrival model of the requests  and fairness considerations. Concerning the arrival model, real-world applications of~\Problem are typically online, in that the type of an incoming request (defined by its group size and itinerary) is only revealed upon its arrival, and reject-or-assign  decisions are irrevocable. We study two online models for~\Problem, defined based on the absence or not of information about  the arrival rates of the request types. From our discussions with practitioners, demand in the Shinkansen system is robust and stable out of peak seasons (between Christmas and New Year's Day, the cherry blossom season in April, and the fall foliage season in November). This information is relevant to our work because we explore the stability and predictability of demand in our online models. Moreover, \Problem is structurally similar to the vector multiple knapsack problem (VKMP), a classic combinatorial optimization problem, so some of our results are of independent interest. We also study the offline~\Problem, in which we assume that all requests are known;  efficient exact algorithms for the offline~\Problem are essential building blocks for online algorithms studied in this paper.

We also investigate variants of~\Problem incorporating fairness conditions. These scenarios deserve special attention, as empirical evidence has shown the challenges of adapting fully utilitarian solution techniques to scenarios where fairness constraints must be observed  (see, e.g.,~\citet{xie2013dynamic} for an example involving an implementation of the formulation presented by~\citet{ciancimino1999mathematical} at the Shenyang Railway Bureau). We study the \textit{first-come, first-serviced} (FCFS) definition of fairness, in which an incoming request \textit{must} be serviced if the train has enough residual seating capacity to accommodate it (i.e., the railway company cannot unilaterally reject an ``inconvenient'' request). FCFS is commonly adopted in practice (JR is one example), and its adoption imposes significant challenges to assignment policies. We also study~\Problem under a \textit{strict first-come, first-serviced} (SFCFS) notion of fairness, where assignment decisions can be modified, but an incoming request must be accepted if it can be part of an assignment plan that includes all previously accepted requests. SFCFS provides higher flexibility to the railway company (in that assignment plans may be modified over time), and passengers are more likely to have their requests assigned, so we study the technical and financial viability of adopting these fairness conditions in practice.

\subsection{Technical Background and Challenges}

\Problem can be cast as a multiple multidimensional knapsack problem, where each knapsack is a coach, and the dimensions represent legs. In particular, \Problem is structurally similar to VKMP because all coaches are identical, i.e., the resource consumption and collected revenue of a request depends only on its type (group size and itinerary) and not on the coach it is assigned to  (see~\cite{naori2019online}). %From the theoretical standpoint, the offline VKMP is challenging; 
\citet{chekuri2004multidimensional} show that the 
offline VKMP
%problem 
does not admit a polynomial-time~$(d^{1-\epsilon})$-approximation algorithm for any~$\epsilon > 0$, where~$d$ is the number of legs. VKMP is also computationally challenging, so previous work has focused mainly on heuristic algorithms (\cite{ahuja2005very,ang2007model,camati2014solving}).

The results and performance guarantees for online problems depend on the model's assumptions. Worst-case competitive analysis, the most popular technique for analyzing online algorithms (\cite{borodin2005online}), assumes that an adversarial environment can choose the parameters of the requests and their arrival order. Under these conditions, packing problems typically do not admit algorithms with constant-factor performance guarantees (\cite{marchetti1995stochastic}). An alternative setting is the random order model, in which an adversarial environment can still choose the parameters of the requests, but the arrival order is randomized, and the number of incoming requests is known a priori; an  example of an online problem in the random order model is the secretary problem (\cite{ferguson1989solved}). The state-of-the-art result for the online VKMP in the random order model is the $(4d + 2)$-competitive algorithm by~\citet{naori2019online}, which is also shown to be asymptotically optimal. Therefore, without additional assumptions, one should not expect to eliminate the dependence on~$d$ (number of legs) from the competitive factor of any algorithm for the online~\Problem unless P = NP. 

The barrier imposed by the number of legs on offline and online multiple multi-dimensional knapsack problems reflects a pathological case with a simple interpretation in the case of~\Problem. Namely, even if a train ride passes through numerous legs, the arrival of many requests with itineraries passing through the same leg creates a bottleneck that hinders the assignment of requests in the future.  More precisely, once a single leg is sold out, all future incoming requests with itineraries passing through that leg must be rejected. This is particularly problematic for the railway company if the accepted requests cover only a few legs. These losses are particularly hard to mitigate in online settings because demand is uncertain and the algorithms do not know whether more profitable requests will arrive in the future.

\subsection{Contributions of this work}

We propose mixed-integer linear programming formulations to solve the offline versions of~\Problem. In particular, we introduce families of inequalities to enforce the FCFS fairness conditions and present a branch-and-cut algorithm to solve the problem. We also present a formulation for offline and online versions of~\Problem
that explores knowledge about arrival rates to pre-assign requests based on their arrival probabilities.

We also 
%investigate online versions of~\Problem, which are more likely to appear in practice. Theoretical guarantees are weak and unrealistic in the more traditional online model, where an adversarial environment has full control over the requests (features, arrival order), and the online algorithm has no information about them. Therefore, we 
propose algorithms for two online models of~\Problem-FCFS and study their theoretical properties. The first consists of the online \Problem-FCFS in the random order model, in which we know the number of incoming requests and assume that the arrival order is randomized. We propose an algorithm for this model and present a refined analysis to derive guarantees for the relevant scenarios of~\Problem. These  results are of independent interest as they extend to the online VKMP. We also present an asymptotically optimal policy for the online \Problem when the algorithm has information about the arrival rates of each request type. 

We conclude the paper with an extensive numerical study of the algorithms proposed in the paper using data extracted from the Shinkansen Tokyo-Shin-Osaka line. In particular, our results suggest that SFCFS fairness conditions are viable both from the a technical and an economical standpoint, i.e., the problem can be solved efficiently and the impact of increased fairness conditions on revenues is negligible.

\section{Literature Review}

The revenue management literature addressed numerous challenges in railway operations, such as pricing and capacity allocation (\cite{manchiraju2022dynamic}), seat assignment (\cite{hetrakul2014latent})
crew management (\cite{caprara1997algorithms}), line planning (\cite{borndorfer2007column}), auction of railway resources (\cite{borndorfer2006auctioning}),
timetabling, rolling stock circulation (\cite{alfieri2006efficient}). For an overview of the operations research literature on railway operations, we refer to~\citet{huisman2005operations}.

\citet{boyar1999seat} is one of the early articles addressing the seat assignment problem with fairness constraints. The paper focuses on the online version of the problem and shows that any fair deterministic algorithm has competitive ratio~$\Theta(1/d)$, where~$d$ is the number of legs (see also~\cite{miyazaki2010improving}); this result is aligned with the findings of~\citet{naori2019online} for the online VMKP. \citet{boyar1999seat} also introduce the concept of~$c$-accommodating algorithms for the seat reservation problem, which indicates the competitive ratio of an online algorithm under the assumption that an optimal solution services all incoming requests. For the special case where ticket prices are not affected by the length of the itinerary, any algorithm is at least~$\frac{1}{2}$-accommodating (i.e., any fair online algorithm can service at least half of the requests).  Regarding upper bounds, \citet{bach2003tight} show that all deterministic algorithms are~$\frac{1}{2}$-accommodating and that the asymptotic upper bound for randomized algorithms is~$\frac{7}{9}$. \citet{boyar2008relative} later compare online algorithms for the seat reservation problem using the relative worst order ratio, which compares the worst-case performance of two algorithms over a restricted set of instances. 

The literature also shows that flexibility in seat assignments provides some advantages. For example, \citet{boyar2004seat} investigate a version of the problem where the reserved seats may be changed (i.e., similar to \Problem with SFCFS fairness conditions), and they show that any online algorithm that changes two or more seats is~$\frac{3}{4}$-accommodating. In a related line of work, \citet{frederiksen2003online} propose a data structure that solves an online version of the problem where the decisions about accepting or rejecting the requests are in real-time, but the seat assignments are only made shortly before the trip starts. 

Another relevant reference is \citet{clausen2010off}, which studies the offline group seat reservation problem. In addition to the capacity planning version of the problem, where the goal is to minimize the number of wagons needed to service all requests, the authors present upper bounds and an exact branch-and-bound algorithm for the case where the goal is to maximize utilization. \citet{deplano2019offline} propose a GRASP-based algorithm for an extension where the seats may have heterogeneous (but fixed) costs.

\Problem has connections with other classical combinatorial optimization problems. For example, \Problem is related to packing rectangles into rectangles, with the dimensions representing seats and legs. Similar problems are observed in scheduling for cluster computing and dynamic storage systems (\cite{leung1990packing,kierstead1991polynomial}). Seat reservation problems involving groups have also been investigated in other settings, such as even seating  (\cite{bergman2019exact, olivier2021quadratic, mehrani2022models}).

\citet{manchiraju2022dynamic}
investigate policies for joint pricing and capacity planning in systems with discrete capacity unities and with congestion management when these capacities are ``soft'' constraints. In contrast with~\citet{manchiraju2022dynamic}, we do not consider capacity planning or pricing, but we allow group requests and forbid overbooking. 
Our model is similar to the one studied by~\citet{bai2023fluid}, but \Problem has a multi-knapsack aspect that is absent from~\citet{bai2023fluid}. Therefore, in addition to accepting requests, we must also choose which coach to assign to a request. From the technical standpoint, we focus on the competitive analysis of our online algorithms instead of regret, which is frequently used in similar problems incorporating dynamic pricing (see, e.g.,~\citet{jasin2014reoptimization,vera2021online,wang2022constant}).

\section{Problem Description and Notation}

Let~$\notation{\network} {(\stations,\setLegs)}$ denote the graphical representation of a railway transportation network. Each vertex in~$\stations$ is associated with a station, and each edge in~$\setLegs$ represents a leg, i.e., a direct (non-stop) connection between a pair of stations. We assume, for simplicity of exposition and without loss of generality, that~$\network$ is a path graph so that there is an order of the stations in~$\stations$ induced by their relative positions; we use~$i \in [|\stations|]$ to denote the~$i$-th station in~$\stations$, and~$\station < \station'$ indicates that stations~$\station$ is located before~$\station'$. The set of legs~$\setLegs$ is given by all pairs of consecutive stations in the sequence, i.e., $\notation{\setLegs}{\{ (\station,\station+1) \,|\, \station \in [|\stations|-1]\}}$. We use the itinerary of the Tokyo-Shinosaka line in our study, which consists of five stations and four legs (see Figure~\ref{fig:shinkansen_map}).

In each \textit{full trip}, a train goes from the first to the last station, thus traversing each leg and itinerary exactly once. \Problem models the problem from the perspective of a railway company that wishes to determine the reject-or-assign policy for ticket requests associated with a single full trip, i.e., there is just one train and one choice of departure time for each itinerary. Therefore, our model does not consider cannibalism effects across different offerings (e.g., other similar full trips with different departure times). We define a discrete planning (or selling) horizon~$\emphmath{\timehorizon}$ to represent  the period when the tickets are available for purchase. In practice, the selling horizon for Shinkansen trains typically consists of 30 days.

A train comprises a set~$\emphmath{\setCoaches}$ of coaches (or wagons), each with a \textit{seating capacity}~$\emphmath{\reservedcapacity} \in \mathbb{N}$; an average Shinkansen train has 16 coaches, each with approximately 100 seats. Seats are assigned to requests over time, so we refer to the \textit{residual (seating) capacity}  when alluding to the number of seats available at a certain moment in the selling horizon. Our model does not consider cancellations, no-shows, or overbooking, and the seating capacity~$\reservedcapacity$ is a hard constraint (i.e., all passengers must have a seat). 

We categorize requests into \textit{types} based on \textit{group size} and \textit{itinerary}. Namely, a request of type~$\emphmath{\requesttype}$ is associated with a group of~$\emphmath{\numPass(\requesttype)}$ passengers traveling together from station~$\emphmath{\origStation(\requesttype)} \in \stations$ to  station~$\emphmath{\destStation(\requesttype)}$; we use~$\emphmath{\requesttypes}$ to denote the set of request types. The set of legs traversed in the itinerary of request type~$\requesttype$ is denoted by~$\emphmath{\setLegs(\requesttype)}$,
and~\notation
{\delta}
{\max\limits_{\requesttype \in\requesttypes}\frac{\numPass(\requesttype)}{\reservedcapacity}} is the maximum fraction of a coach's seating capacity that a group may occupy. In practice, $\numPass(\requesttype) \leq 6$, so~$\delta \approx 0.06\%$; the fact that~$\delta$ is small (i.e., the maximum group size is significantly smaller than the seating capacity) plays an important role in our analysis.

We use~$\emphmath{\requests}$ to denote the set of requests and~$\notation{\request}{ (\requesttype,j)}$ to indicate that~$\request$ is the~$j$-th incoming request of type~$\requesttype$.  Given two requests~$\request$ and $\request'$, the relation~$\request < \request'$ indicates that~$\request$ arrives before~$\request'$. For~$\timetogo$ in~$\mathbb{N}$, \emphmath{\requests_\timetogo} contains the first~$\timetogo$ incoming requests of~$\requests$, and \emphmath{\request_{\timetogo}} denotes the~$\timetogo$-th incoming request. The revenue collected from each assigned request of type~$\requesttype$ is denoted by~$\emphmath{\price(\requesttype)}$; salvage values are zero, i.e., the railway company does not make (or save) money from unsold seats. We use~$\emphmath{\requesttype(\request)}$ to denote the type of request~$\request$, and to simplify the notation, we sporadically interchange~$\requesttype(\request)$ with~$\request$ (e.g., $\numPass(\request)$  represents~$\numPass(\requesttype(\request))$.

We use~\notation{\assignments}{ \requests \times \coaches} (or, equivalently, $\assignments \coloneqq \requesttypes \times \mathbb{N} \times \coaches$) to denote the set of all possible assignments of requests to coaches; namely, a pair~$(\request,\coach)$ (triple~$(\requesttype,k,\coach)$) in $\assignments$ indicates that request~$\request$ (the~$k^{th}$ incoming request of type~$\requesttype$)   is assigned to coach~$\coach$. An \textit{assignment plan}~$\notation{\mathcal{A}}{ \{(\request,\coach) \in \requests(\mathcal{A}) \times \coaches\}}$ is a subset of~$\assignments$ representing the assignments of each request~$\request$ in~$\requests(\mathcal{A})$ to exactly one coach.  We use~\emphmath{\mathcal{A}_\timetogo} to denote the restrictions of~$\mathcal{A}$ to the first~$\timetogo$ incoming requests. In a feasible assignment plan, the number of passengers assigned to a coach in any given leg must not exceed the seating capacity~$\reservedcapacity$. Additionally, requests cannot be fragmented or only partially satisfied; if a request~$\request$ is accepted, all~$\numPass(\request)$ passengers must be assigned to the same coach for the entire requested itinerary. The goal of the railway company is to identify a reject-or-assign policy that identifies a feasible assignment plan~$\mathcal{A}$ such that~\notation{\profit(\mathcal{A})}{\sum\limits_{\request \in \requests(\mathcal{A})}\price(\requesttype)} (or just~$\profit$ if~$\mathcal{A}$ is implicit) is maximum. We use the acronym~\Problem when referring to this baseline version of the problem.

\subsection{Fairness Constraints}\label{sec: fairness notation}

Some railway companies must adhere to \textit{fairness constraints} that forbid unilateral request rejections. We use~$\emphmath{\veccapacity_\timetogo} \in \mathbb{N}^{\setCoaches \times \setLegs}$ to denote the train's (residual) capacity upon the arrival of request~$\request_{\timetogo}$; in particular, $\veccapacity_1$ denotes the original seating capacity. We drop the time-step index from~$\emphmath{\veccapacity_\timetogo}$ and use~$\veccapacity$ instead whenever~$\timetogo$ is clear or irrelevant in the context. We use~$\emphmath{\seatdemand^{\coach,  \requesttype}} \in \mathbb{N}^{\setCoaches \times \setLegs}$ to represent the resource consumption resulting from the assignment of a request of type~$\requesttype$ to coach~$\coach$, i.e., $q^{\coach,  \requesttype}_{\coach,\leg} = \numPass(\requesttype)$ if~$\leg \in \setLegs(\requesttype)$ and $q^{\coach,  \requesttype}_{\coach',\leg'} = 0$ if~$\coach' \neq \coach$ or~$\leg' \notin \setLegs(\requesttype)$. Therefore, given a residual capacity matrix~$\veccapacity$, assigning a request of type~$\requesttype$ to coach~$\coach$ yields the residual capacity~$\veccapacity - \seatdemand^{\coach,\requesttype}$. We use~$\veccapacity - \seatdemand^{\coach,  \requesttype} \geq 0$ to indicate that the residual capacity of coach~$\coach$ for each leg in~$\setLegs(\requesttype)$ is not smaller than~$\numPass(\requesttype)$, i.e., $\veccapacity - \seatdemand^{\coach,  \requesttype} \geq 0$ indicates that a request of type~$\requesttype$ can be assigned to~$\coach$. We use \notation{\coaches_{\veccapacity,\requesttype}}{ \{ \coach \in \coaches: \veccapacity - \seatdemand_{\coach,  \requesttype} \geq 0 \} } to denote the set of coaches that can be assigned to a request of type~$\requesttype$ given residual capacity~$\veccapacity$.

\subsubsection{First-Come, First Serviced}

FCFS fairness conditions force the assignment of any incoming request~$\request_\timetogo$ in situations where the residual capacities are sufficiently large, i.e., if~$\coaches_{\veccapacity_\timetogo,\request_\timetogo} \neq \emptyset$. In particular, the reject-or-assign decisions of JR (the operator of the Shinkansen trains) are subject to FCFS constraints. We use the acronym~\emphmath{\Problem\text{-FCFS}} to denote~\Problem with FCFS constraints. From a technical standpoint, \Problem-FCFS is significantly more challenging than~\Problem, and the economic implications of imposing  can be significant, as exhibited in Example~\ref{ex: fairness}. 

\begin{example}[Price of fairness]\label{ex: fairness}
Figure~\ref{fig:fairness} depicts the residual capacity~$\request_\timetogo$ of a coach for each leg (e.g., there are 22 seats available in the leg connecting Nagoya to Shin-Yokohama) and the resource demands of requests~$\request_{\timetogo}$ and~$\request_{\timetogo+1}$. Request~$\request_\timetogo$ is associated with a person who wants to travel from Shin-Osaka to Kyoto; as there are six seats available in this leg upon the arrival of~$\request_\timetogo$, request~$\request_\timetogo$ must be assigned according to the FCFS constraint. As a result, the residual capacity~$\veccapacity_{\timetogo+1}$ in the first leg goes down to 5, so request~$\request_{\timetogo+1}$, associated with six passengers willing to travel through all legs, must be rejected.
\end{example}

\begin{figure}
    \centering
\begin{tikzpicture}[transform shape, scale=0.65, every node/.style={minimum size=1cm, inner sep=0pt}, every label/.style={font=\scriptsize},
box/.style={draw, rounded corners, inner sep=10pt, align=center}]

    % First path graph
    % \node[box] (box1) at (4,0) {
    \node[text width=3cm] at (-3,0) 
    {\scriptsize Residual  Capacity~$\veccapacity_\timetogo$};
    \draw[draw=black] (-1.5,-1.25) rectangle ++(13,2.5);
        \node [circle, draw, minimum size=2cm, inner sep=0pt] (A1) at (0,0) {\scriptsize Shin-Osaka};
        \node [circle, draw, minimum size=2cm, inner sep=0pt] (A2) at (2.5,0) {\scriptsize Kyoto};
        \node [circle, draw, minimum size=2cm, inner sep=0pt] (A3) at (5.0,0) {\scriptsize Nagoya};
        \node [circle, draw, minimum size=2cm, inner sep=0pt] (A4) at (7.5,0) {\scriptsize Shin-Yokohama};
        \node [circle, draw, minimum size=2cm, inner sep=0pt] (A5) at (10,0) {\scriptsize Tokyo};
    
        \draw (A1) -- node[above] {6} (A2);
        \draw (A2) -- node[above] {7} (A3);
        \draw (A3) -- node[above] {22} (A4);
        \draw (A4) -- node[above] {9} (A5);
    % };
    % \node at (4,1.5) {Box 1};

    % Second path graph
    \node[text width=3cm] at (-3,-3) {\scriptsize Request~$\request_{\timetogo}$};
    \draw[draw=black] (-1.5,-4.25) rectangle ++(13,2.5);

    \node [circle, draw, minimum size=2cm, inner sep=0pt] (B1) at (0,-3) {\scriptsize Shin-Osaka};
    \node [circle, draw, minimum size=2cm, inner sep=0pt] (B2) at (2.5,-3) {\scriptsize Kyoto};
    \node [circle, draw, minimum size=2cm, inner sep=0pt] (B3) at (5,-3) {\scriptsize Nagoya};
    \node [circle, draw, minimum size=2cm, inner sep=0pt] (B4) at (7.5,-3) {\scriptsize Shin-Yokohama};
    \node [circle, draw, minimum size=2cm, inner sep=0pt] (B5) at (10,-3) {\scriptsize Tokyo};

    \draw (B1) -- node[above] {\textcolor{blue} 1} (B2);
    \draw (B2) -- node[above] {\textcolor{blue} 0} (B3);
    \draw (B3) -- node[above] {\textcolor{blue} 0} (B4);
    \draw (B4) -- node[above] {\textcolor{blue} 0} (B5);

    % Third path graph
    \node[text width=3cm] at (-3,-6) {\scriptsize Request~$\request_{\timetogo+1}$};
    \draw[draw=black] (-1.5,-7.25) rectangle ++(13,2.5);
    \node [circle, draw, minimum size=2cm, inner sep=0pt] (C1) at (0,-6) {\scriptsize Shin-Osaka};
    \node [circle, draw, minimum size=2cm, inner sep=0pt] (C2) at (2.5,-6) {\scriptsize Kyoto};
    \node [circle, draw, minimum size=2cm, inner sep=0pt] (C3) at (5,-6) {\scriptsize Nagoya};
    \node [circle, draw, minimum size=2cm, inner sep=0pt] (C4) at (7.5,-6) {\scriptsize Shin-Yokohama};
    \node [circle, draw, minimum size=2cm, inner sep=0pt] (C5) at (10,-6) {\scriptsize Tokyo};

    \draw (C1) -- node[above] {\textcolor{red} 6} (C2);
    \draw (C2) -- node[above] {\textcolor{blue} 6} (C3);
    \draw (C3) -- node[above] {\textcolor{blue} 6} (C4);
    \draw (C4) -- node[above] {\textcolor{blue} 6} (C5);

\end{tikzpicture}
    \caption{Example of adversarial setting for FCFS fairness. 
    }
    \label{fig:fairness}
\end{figure}
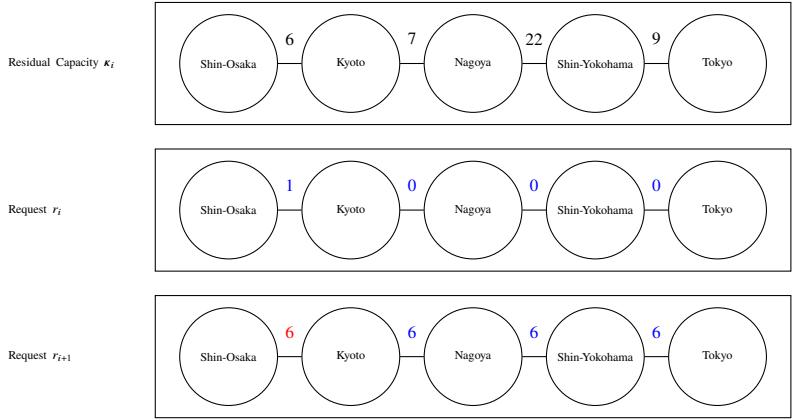

\subsubsection{Strict First-Come, First Serviced}

We also study~\emphmath{\Problem\text{-SFCFS}}, a version of~\Problem that adopts a stricter notion of fairness while providing more  flexibility for the assignment of requests.  In this setting, the railway company only needs to maintain the set~$\requests(\mathcal{A}_{\timetogo})$ of assigned requests, whereas the assignments of requests to coaches may change over time.
%, i.e., $\requests(\mathcal{A}_{\timetogo})$ is not necessarily a subset of~$\requests(\mathcal{A}_{\timetogo+k})$ for~$k \geq 1$. 
In exchange for the flexibility gained from delayed assignments, the railway company must accept an incoming request~$\request_\timetogo$ if there exists \textit{any} feasible assignment~$\mathcal{A}_{\timetogo+1}$ such that~$\requests(\mathcal{A}_{\timetogo}) = \requests(\mathcal{A}_{\timetogo}) \cup \{\request_{\timetogo}\}$.
%, i.e., $\mathcal{A}_{\timetogo+1}$ assigns~$\request_\timetogo$ and all requests that have been accepted in previous steps. 

\subsection{Online Extensions} \label{sec: online notation}

In online settings of~\Problem, the requests only become known upon arrival, and reject-or-assign decisions about a request~$\request_\timetogo$ are irrevocable and must be made before~$\request_{\timetogo+1}$ arrives. We assume that each time step in the planning horizon~$\timehorizon$ is associated with exactly one request, so~$\timetogo \in \timehorizon$ denotes a time step and and an arrival order. We study two online models for~\Problem, distinguished based on the information available to the railway company. 

\paragraph{Random Order Model}
In the random order model, we are given the number~\notation{\numrequests}{|\requests|} of arrivals. This assumption is important for the theoretical analysis and  mild for real-world scenarios, as demand for seats in Shinkansen trains is stable, and the railway company can work with adequate estimates of~$\numrequests$. For our analysis, we assume that the features of the requests (number of passengers, itinerary, and fare) can be adversarially  chosen, but the arrival order is uniformly sampled from all~$\numrequests!$ permutations of~$\requests$. 

This randomized setting provides a more realistic representation of worst-case scenarios, as fully adversarial environments (e.g., scenarios where an adversary controls the requests and their arrival order) are overly pessimistic and unlikely to be observed in practice. Moreover, the theoretical guarantees of algorithms in this model are typically stronger and better reflect the performance of high-quality policies in practice. The expected competitive ratio of an online algorithm in the random order model is defined as the ratio between the expected objective attained by the algorithm and the optimal objective value~\emphmath{\profit^*} of the offline version of the problem for the same instance.

\paragraph{Arrival Rate}
We also consider an online model where the railway company knows the arrival rate~\emphmath{\arrivalrate_{\requesttype}} $> 0$ of each request type~$\requesttype$. In this model, the number of arrivals is uncertain, so we use~$\emphmath{X}$ to denote the random variable representing the number of requests. We assume that~$Pr[X \geq 1] = 1$, i.e., at least one request arrives over the entire sale horizon almost surely (as the problem is trivial otherwise). We restrict our analysis to the arrivals happening within a time horizon~\emphmath{\timehorizon}, which 
%We assume a one-to-one relationship between time steps and arrivals, so our assumption 
implies a restriction on the number of requests we consider; in particular, we have~$Pr[X > |\timehorizon|] > 0$. We note that the choice of~$\timehorizon$ is relevant only for our analysis and not for our policies. We use \notation{\survival_{\timetogo}}{ Pr[X \geq \timetogo + 1 |X \geq \timetogo ] } to denote  the probability that at least~$\timetogo+1$ requests will arrive given that~$\timetogo$ requests have been observed. 

\newcommand{\probability}{\ensuremath{\rho}}

We assume the requests follow a Markovian arrival process, i.e., all requests arrive independently. Therefore, the arrival of requests can be interpreted as the combination of multiplied independent Poisson processes, so a request of type~$\requesttype$ arrives in any time step with probability~
${\arrivalrate_{\requesttype}}/{\sum\limits_{\requesttype \in \requesttypes}\arrivalrate_{\requesttype}}$.
% $\frac{\arrivalrate_{\requesttype}}{\sum\limits_{\requesttype \in \requesttypes}\arrivalrate_{\requesttype}}$. 
For ease of exposition, we assume that the arrival rates are scaled to the time intervals defining the time steps so that~$\sum\limits_{\requesttype \in \requesttypes}\arrivalrate_{\requesttype} = 1.0$, i.e., $\arrivalrate_{\requesttype}$ represents the arrival rate and the arrival probability for requests of type~$\requesttype$. The Markovian assumption about arrival rates is not adequate in general, as it assumes limitations on the variability in demand that is not always observed in practice (\cite{talluri2006theory, aouad2022nonparametric,bai2023fluid}). Nevertheless, demand is stable in the Shinkansen system, so this assumption is adequate for our setting.

\section{Online Policies for \Problem}
\label{section:Online Policies for CRP}

This section studies~\Problem, the version of the problem where the railway company does not have to comply with fairness constraints, i.e., it can unilaterally reject requests. We present an exact mixed-integer programming (MIP) formulation of~\Problem for the offline case and study algorithms for two online versions.

\subsection{Mathematical Programming Formulation for the Offline~\ProblemNoFairness}

We present~\ref{model:offline} below, an exact MIP formulation that solves~\emphmath{\Problem}.
\begin{align}\label{model:offline}
\tag{\Problem(\requests)}
    % \max\  &\sum\limits_{(\request,\coach) \in \mathcal{A}}
    % \price(\requesttype) \cdot  x_{\requesttype,j,\coach} &  \\
    \max\  &\sum\limits_{\request \in \requests}
    \price(\requesttype) \cdot  y_{\request} &  \\
&
    \sum\limits_{\coach \in \coaches}  x_{\request,\coach} = y_{\request}, 
    % & 
    \forall \request  \in \mathcal{R}  \nonumber \\
%(\beta) 
&
    \sum\limits_{\substack{(\request,\coach) \in \assignments \\ 
    \leg \in \setLegs(\requesttype)}} \numPass(\requesttype) \cdot x_{\request,\coach} \leq \reservedcapacity,
    % & 
    \forall (\leg,\coach) \in  \setLegs \times \coaches  \nonumber \\    
%(\gamma) 
& 
    y_{\request}  \in \{0,1\}, x_{\request,\coach} \in \{0,1\}, 
    % &
    \forall (\request,\coach) \in \assignments \nonumber
\end{align}
% \]
Formulation~\ref{model:offline} uses binary variable~\emphmath{x_{\request,\coach}} to represent the assignment of request~$\request$ to coach~$\coach$, and binary variables~\emphmath{y_{\request}} to indicate that request~$\request$ has been assigned. The objective function represents the revenue collected from all assigned requests. The first set of constraints synchronizes~$\boldsymbol{x}$ and~$\boldsymbol{y}$ and forbids multiple assignments of the same request, i.e., $y_{\request}$ is activated if and only if~$x_{\request,\coach} = 1$ for exactly one~$\coach$ in~$\coaches$. The second set of constraints controls the capacity utilization of each coach per leg. Variables~$\boldsymbol{y}$ are redundant, but we incorporate them into~\ref{model:offline} because we later extend this formulation to solve~\Problem-FCFS.

\subsection{\Problem in the Random Order Model}\label{sec:random order model}

\citet{naori2019online} present an asymptotic $\frac{1}{4|\setLegs|+2}$-competitive algorithm for~\Problem in the random order model and show that this ratio is asymptotically optimal, i.e., without additional assumptions, we cannot find an algorithm that eliminates the dependence on~$\frac{1}{|\setLegs|}$. 
%This barrier emerges when some legs are sold out and become a bottleneck. In extreme cases, the railway company may find itself in a situation where the residual capacity is still very large (i.e., there are many seats available for almost all legs), but it cannot sell more tickets because most itineraries include a sold-out leg.
However, \citet{naori2019online} must consider scenarios where the maximum fraction~$\delta$ that an item may occupy from the total packing capacity in any given dimension is large  (e.g., $\delta \approx 1.0$); such cases are not observed in practical instances of~\Problem, where~$\delta \approx 0.06$.
%; recall that, in our case, the largest group consists of six passengers, whereas each coach has around 100 seats available. 
We study and analyze an alternative algorithm 
%for the online~\Problem 
that explores knowledge about~$\delta$ to obtain a stronger asymptotic competitive ratio for the cases where~$\delta$ is small; in particular, we obtain a competitive ratio that does not depend on~$|\setLegs|$ and is larger than $\frac{1}{4|\setLegs|+2}$ for~$\delta \leq 0.9$. 

\subsubsection{Description of the Algorithm}
Algorithm~\ref{alg:NR} describes~$\texttt{NR}(q)$, an iterative two-phase procedure to solve~\Problem in the random order model. $\texttt{NR}(q)$ relies on the solution of~$\Problem(\requests_\timetogo)_{LR}$, the linear programming (LP) relaxation of~$\Problem(\requests_\timetogo)$, to make reject-or-assign decisions about~$\request_\timetogo$. The parameter~$q$ is a number in~$[0,1]$ that defines the fraction of~$\requests$ handled in each phase, as described next.

\begin{algorithm}
\begin{algorithmic}[1]
\footnotesize
\State set~$\vecx = \boldsymbol{0}, \veccapacity = \veccapacity_1$;
% \For{every incoming request~$\request_\timetogo$ such that $1 \leq \timetogo \leq q\numrequests$}
% \State  reject~$\request_\timetogo$;
% \EndFor{}
\For{every incoming request~$\request_\timetogo$}
\If{ $\timetogo > q\numrequests$ and  $\coaches_{\veccapacity,\request_\timetogo} \neq \emptyset$  }
% \State reject $\request_\timetogo$;
% \Else{}
\State  obtain an optimal solution~$\vecx^{(\timetogo)}$ to $\Problem(\requests_\timetogo)_{LR}$;
\State sample value $z \in [0,1]$ uniformly at random;
\If{$\sum\limits_{\coach \in \setCoaches} x^{(\timetogo)}_{\request_\timetogo,\coach} \geq z$  }
% \If{$\vecx_{\request_\timetogo,\coach} = 1$ for some $\coach$ in $\coaches$}
\State set $\vecx_{\request_\timetogo,\coach} = 1$ and $\veccapacity = \veccapacity - \seatdemand^{\coach,  \request_{\timetogo}}$
for $\coach = \min(\coaches_{\veccapacity,\request_{\timetogo}})$;
%the smallest~$\coach$ such that $\vecx_{\request_\timetogo,\coach} = 1$ is feasible;
\EndIf{}
% \Else
% \State reject $\request_\timetogo$;
\EndIf{}
\EndFor{}
\end{algorithmic}
\caption{$\texttt{NR}(q)$}
\label{alg:NR} 
\end{algorithm}

$\texttt{NR}(q)$ uses~$\vecx \in \{0,1\}^{\requests \times \coaches}$ and~$\veccapacity$ to bookkeep the assignments and the residual capacity, respectively. The first~$q\numrequests$ requests arrive during the \textit{sampling phase}, in which $\texttt{NR}(q)$ rejects all requests. After the sampling phase, $\texttt{NR}(q)$ proceeds with the \textit{packing phase}, in which the reject-or-assign decision for each request~$\request_\timetogo$ is based on a  solution for~$\Problem(\requests_\timetogo)_{LR}$. Requests that cannot be assigned because the residual capacity is insufficient (i.e., if $\coaches_{\veccapacity,\request_{\timetogo}} = \emptyset$) are rejected. Otherwise, $\texttt{NR}(q)$ adopts a randomized approach, which allows it to avoid low-utility requests: $\texttt{NR}(q)$ assigns~$\request_\timetogo$  to the  coach with the smallest index~$\coach$ in~$\coaches_{\veccapacity,\request_{\timetogo}}$  with probability~$\sum_{\coach \in \setCoaches}  x^{(\timetogo)}_{\request_\timetogo,\coach}$,
and rejects~$\request_\timetogo$ with probability $1 - \sum_{\coach \in \setCoaches}  x^{(\timetogo)}_{\request_\timetogo,\coach}$, $\texttt{NR}(q)$. 

\subsubsection{Analysis of $\texttt{NR}(q)$}

Let~\emphmath{\requests^{(\timetogo)}} be the subset of requests in~$\requests_\timetogo$ assigned by $\texttt{NR}(q)$
and~\emphmath{O^{(\timetogo)}_{\leg}} be the number of occupied seats for leg~$\leg$  across all coaches. For ease of notation, we assume that the original seating capacity~$\reservedcapacity$ per leg and coach is normalized to 1, so the maximum resource consumption of a request per leg is~$\delta$, and the total seating capacity for a given leg (across all coaches) is~$|\coaches|$. Remark~\ref{remark:blocking} characterizes a sufficient condition on the current residual capacity that enables the assignment of \textit{any} incoming request.

\begin{remark}[Safe assignments]\label{remark:blocking} Request~$\request_{\timetogo}$ can be assigned if 
\notation{\overline{O}^{(\timetogo)}}{\max\limits_{\leg \in \setLegs}O^{(\timetogo)}_{\leg}}
$  \leq (1 - \delta) |\coaches|$.
\end{remark}
\proof{Proof:} 
The result follows directly from the pigeonhole principle, as at least one coach must exist with sufficient residual capacity across all legs in~$\setLegs(\request_{\timetogo})$ if the inequality holds. %$\hfill \blacksquare$ 
\medskip
\endproof

Lemma~\ref{lemma:Expected Optimal Utility}  provides a lower bound on the expected revenue~$\mathop{\mathbb{E}}[\price(\request_\timetogo)]$ gained from any arbitrary request~$\request_\timetogo$. From the uniformity of the arrival orders, this fraction is identical (in expectation) across all requests. 
\begin{lemma}[Expected Utility]\label{lemma:Expected Optimal Utility}
    The expected revenue collected from request~$\request_\timetogo$ to OPT is $            \mathop{\mathbb{E}}[ \price(\request_{\timetogo})]
        \geq 
        \frac{OPT}{n}$.
\end{lemma}
\proof{Proof:}
Let~$\requests^*$ be the set of assigned requests in an arbitrary optimal solution to~\Problem, i.e., $\requests^*$ attains~$OPT$. From the uniformity of the arrival orders, it follows that for every~$\timetogo$ in~$[\numrequests]$, a fraction~$\frac{\timetogo}{n}$ of the elements in~$\requests^*$ belong to~$\requests_\timetogo$ in expectation. Moreover, in expectation, the objective value attained by requests in~$\requests_\timetogo \cap \requests^*$ is~$\frac{\timetogo}{n}OPT$. Lastly, from the uniformity assumption, each request in~$\requests_\timetogo$ is the~$\timetogo$-th incoming one with probability~$\frac{1}{\timetogo}$, and by combining those expectations, we conclude that the expected contribution of~$\request_\timetogo$ to~$OPT$ is~$\frac{1}{n}OPT$. We note that a similar result has been proven by~\citet{kesselheim2018primal}.  %$\hfill \blacksquare$
\medskip
\endproof

Lemma~\ref{lemma:Profit per arrival} provides a lower bound for the expected revenue collected from each request assigned during the packing phase. The result requires the assumption that $q \geq 
\frac{|\coaches|}{\delta\numrequests}$, i.e., $\texttt{NR}(q)$ must reject the first~$\frac{|\coaches|}{\delta}$ requests, which is the minimum number of requests needed to sell out all seats for one leg. Intuitively, the idea is that information about a significant number of requests provides sufficient information about the items composing an optimal solution. This assumption is asymptotically mild, though, as the lower bound on the length of the sampling phase depends solely on~$|\coaches|$ and~$\delta$, i.e., it is constant in terms of~$\numrequests$.
\begin{lemma}[Profit per arrival]\label{lemma:Profit per arrival}
If $q \geq 
\frac{|\coaches|}{\delta\numrequests}$,
    the expected revenue collected from~$\request_\timetogo$ for $\timetogo \geq q \numrequests$  is
    \begin{eqnarray*}
            \mathop{\mathbb{E}}[ \price(\request_{\timetogo})]
        \geq \left[1 - \frac{1}{(1 - \delta)}  \ln{\left(\frac{\timetogo-1}{qn}\right)}  \right]     
%        \sum_{k = qn + 1}^{\timetogo-1}\frac{1}{k}\right)
\frac{OPT}{n}.
        \end{eqnarray*}
    % \begin{eqnarray*}
    %         \mathop{\mathbb{E}}[ p_{\timetogo}]
    %     \geq \left(1 - \frac{|\setLegs|}{(1 - \delta)}  \sum_{k = qn + 1}^{\timetogo-1}\frac{1}{k}\right)\frac{OPT}{n}.
    %     \end{eqnarray*}
\end{lemma}
\proof{Proof:}
For each request~$\request_\timetogo$ arriving during the packing phase (i.e., when $\timetogo \geq q \numrequests$), $\texttt{NR}(q)$ obtains an optimal solution~$\vecx^{(\timetogo)}$ to~\textbf{Offline}$(\requests_\timetogo)_{LR}$. From the feasibility of~$\vecx^{(\timetogo)}$ and the uniformity of the arrival sequences, the average utilization of a leg~$\leg$ by request~$\request_\timetogo$ is~$\frac{|\coaches|}{\timetogo}$, which is the ratio between the total (normalized) seating capacity~$|\coaches|$ for leg~$\leg$ and the number of requests in~$\requests_{\timetogo}$. This bound is non-trivial because $q \geq \frac{|\coaches|}{\delta\numrequests}$ and, consequently, $\timetogo \geq \frac{|\coaches|}{\delta}$, as the maximum resource consumption of any request is~$\delta$. This observation holds for any leg, including the one for which the residual capacity is the lowest (i.e., the leg~$\leg$ such that~$\overline{O}^{(\timetogo)} = O^{(\timetogo)}_\leg$). Therefore, the expected value of~$\overline{O}^{(\timetogo)}$ upon the arrival of~$\request^{(\timetogo)}$  is such that
    \begin{eqnarray*}
    \mathop{\mathbb{E}}\left[\overline{O}^{(\timetogo)}\right] 
    \leq  
    \sum_{k = qn + 1}^{\timetogo-1}\frac{|\coaches|}{k}\left(
\sum\limits_{\coach \in \setCoaches} 
x^{(k)}_{\request_k,\coach}
\right)
   \leq
   \sum_{k = qn + 1}^{\timetogo-1}\frac{|\coaches|}{k}
\leq
   |\coaches| \ln{\left(\frac{\timetogo-1}{qn}\right)}.
    \end{eqnarray*}
    The first inequality follows because
$\texttt{NR}(q)$ only incorporates requests arriving during the knapsack phase (i.e., the train is empty upon the arrival of request~$\request_{q\numrequests+1}$). The  second inequality holds because the the assignment probability~$\sum_{\coach \in \setCoaches} x^{(k)}_{\request_k,\coach}$  of~$\request_{\timetogo}$ is not larger than 1. Finally, the transformation used in the last inequality follows from~$\sum\limits_{x=a}^{b}\frac{1}{x} \leq \int\limits_{a-1}^{b}\frac{1}{x} dx$. The combination of the expectation above with Remark~\ref{remark:blocking} and Markov's Inequality shows that there is sufficient residual capacity to assign request~$\request_{\timetogo}$ with probability
    \begin{eqnarray*}
    \Pr\left[\overline{O}^{(\timetogo)} < (1 - \delta) |\coaches|\right] 
            &=& 
                1 - \Pr\left[\overline{O}^{(\timetogo)}\geq (1 - \delta)  |\coaches|\right] 
            \geq
                1 - \frac{|\coaches| \ln{\left(\frac{\timetogo-1}{qn}\right)}}{(1 - \delta)|\coaches|} 
            =
                1 - \frac{1}{ (1 - \delta) }  \ln{\left(\frac{\timetogo-1}{qn}\right)}.
    \end{eqnarray*}
    Finally, the result follows from the combination of the bound above and  Lemma~\ref{lemma:Expected Optimal Utility}. %$\hfill\blacksquare$
    \medskip
\endproof

Theorem~\ref{thm:randomorder} aggregates the expected returns across all arrivals in the packing phase to derive the performance guarantees of~$\texttt{NR}(q)$. Note that the competitive ratio of~$\texttt{NR}(q)$ does not depend on~$|\setLegs|$.

\begin{theorem}\label{thm:randomorder}
    $\texttt{NR}(q)$ is $\left(1 - \frac{\ln{(2-\delta)}}{1-\delta}\right)$-competitive for~$q = \frac{1}{2-\delta}$.
\end{theorem}
\proof{Proof:}
    The expected objective value of the solution identified by the algorithm is the sum of the expected values of each incoming item. By adding up the terms over all~$\timetogo \in [qn+1,n]$, we obtain:
    \begin{eqnarray*}
        \mathop{\mathbb{E}}[p]
            &\geq&
                \frac{OPT}{n}\sum_{\timetogo = qn+1}^n \left( 1 - \frac{1}{(1 - \delta)}  \ln{\left(\frac{\timetogo-1}{qn}\right)} \right)
=
            % \\
            % &=&
                \frac{OPT}{n} \left[ 
                (n - qn - 1) -  \frac{1}{(1 - \delta)} \sum_{\timetogo = qn+1}^n \left(    \ln{\left(\frac{\timetogo-1}{qn}\right)} \right) 
                \right]
            \\
            &=&
                \frac{OPT}{n} \left[ 
                (n - qn - 1) -  \frac{1}{(1 - \delta)} \sum_{\timetogo = qn}^{n-1} \left(    \ln{\left(\frac{\timetogo}{qn}\right)} \right)
                \right].
    \end{eqnarray*}
As~$\ln{\left(\frac{\timetogo}{qn}\right)}$ is an increasing function in~$\timetogo$, we have
\[
\sum\limits_{\timetogo = qn }^{n-1} \ln{\frac{\timetogo}{qn }} \leq 
\int\limits_{\timetogo = qn}^{n} \ln{\frac{\timetogo}{qn}} \text{d}l
    = 
    n\left(\ln{\frac{n}{qn}} - 1\right) - qn\left(\ln{\frac{qn}{qn}} - 1\right) 
    = 
    qn   - n - n\ln{q}.
\]
By replacing the inequality above, we obtain the asymptotic bound
\begin{eqnarray*}
    \lim\limits_{n \rightarrow \infty}
    \mathop{\mathbb{E}}[p]
        &\geq&
        \lim\limits_{n \rightarrow \infty} \left( 
            \frac{OPT}{n} \left[ 
            (n - qn - 1) -  \frac{1}{(1 - \delta)} \sum_{\timetogo = qn}^{n-1} \left(    \ln{\left(\frac{\timetogo}{qn}\right)} \right)
            \right]
            \right)
        \\
        &\geq&
        \lim\limits_{n \rightarrow \infty} \left( 
            \frac{OPT}{n} \left[ 
            (n - qn - 1) +  \frac{1}{(1 - \delta)} \left( n -  qn  + n\ln{q}  \right)         
            \right]
            \right)
        % \\
        % &=&
        %     \lim\limits_{n \rightarrow \infty} \left( 
        %     OPT \left[ 
        %     \left(1 - q - \frac{1}{\numrequests}\right) +  \frac{1}{(1 - \delta)} \left(1 -  q  + \ln{q}  \right)            
        %     \right]
        %     \right)
        % \\
        % &=& 
=
            OPT \left[ 
            \left(1 - q \right) +  \frac{1}{(1 - \delta)} \left(1 -  q  + \ln{q}  \right)            
            \right]
            .
\end{eqnarray*}
The value of~$q$ for which the bound above is maximum is~$q = \frac{1}{2 - \delta}$, which gives  $\lim\limits_{\numrequests \rightarrow \infty}
\mathop{\mathbb{E}}[p]
        \geq
        1 - \frac{\ln{(2-\delta)}}{1-\delta}$. %$\hfill \blacksquare$  
% We maximize the lower bound above by setting~$q = \frac{1}{2 - \delta}$, which gives us the asymptotic bound of the statement. $\blacksquare$
% so we obtain the asymptotic bound
% \begin{eqnarray*}
% \lim_{\numrequests \rightarrow \infty}
% \mathop{\mathbb{E}}[p]
%         \geq
%         1 - \frac{\ln{(2-\delta)}}{1-\delta}. \hfill \blacksquare
% \end{eqnarray*}  
\endproof

For real-world instances, our analysis shows that $\texttt{NR}(q)$ is~$1 - \frac{\ln{(2-\delta)}}{1-\delta} \approx 0.29$-competitive, versus~$\frac{1}{4|\setLegs|+2} = 0.045$ from the analysis by~\citet{naori2019online}. More generally, our analysis gives better competitive ratios for~$\delta \leq 0.9$ (i.e., if individual requests may occupy more than 90\% of a coach's seating capacity).

\subsection{\Problem with Arrival Rates}

 This subsection studies scenarios where the arrival rates of each request type are known. We present a policy derived from a fluid approximation (or deterministic linear relaxation) of a dynamic programming (DP) formulation of~\Problem. We also show that this policy is asymptotically optimal for practical scenarios.

\subsubsection{Exact Dynamic Programming Formulation}

For any incoming request~$\request_\timetogo$ and residual capacity~$\veccapacity_\timetogo$, the optimal assignment decision is given by the following DP, in which the actions (assignments) are represented by~$\emphmath{\boldsymbol{u}} \in \{0,1\}^{\setCoaches \times \timehorizon}$: 
\begin{eqnarray}\label{eq:dpformulation fixed price}
    J(\veccapacity_\timetogo,\timetogo)
        &\coloneqq&
        \max_{u \in \mathcal{U}(\veccapacity_\timetogo)}
        % \left(
            \underbrace{
            \sum_{\requesttype \in \requesttypes}\arrivalrate_{\requesttype}
            }_{\text{Arrival probability}}
            % \left
            (
                \underbrace{
                \sum_{\coach \in \coaches}\price({\requesttype})u_{\coach, \requesttype }
                }_{\text{Assignment decision}}
                +  
                \underbrace{
                \survival_{\timetogo} 
                J^{S}\left(\veccapacity_\timetogo -  \sum_{\coach \in \coaches}u_{\coach, \requesttype}\seatdemand_{\coach,  \requesttype},\timetogo+1\right)
                }_{\text{Revenue for future arrivals with adjusted capacity}}
            % \right
            ).
        % \right).
\end{eqnarray}
The set~$\mathcal{U}(\veccapacity_\timetogo)$ of feasible actions is defined in terms of the residual capacity~$\veccapacity$ as follows: 
\begin{eqnarray}\label{feasible actions}
\mathcal{U}(\veccapacity_\timetogo) 
    \coloneqq 
    % \left\{ 
    \{
    \boldsymbol{u} \in \{0,1\}^{\coaches \times \requesttypes  }: 
    \underbrace{
    q_{\requesttype,\leg}  
    u_{\coach,\requesttype} \leq \reservedcapacity_{\coach,\leg} \text{ for every } \leg \in \setLegs(\requesttype)}_{\text{Capacity constraints}}, 
    \underbrace{\sum_{\coach \in \coaches_{\veccapacity,\requesttype} }u_{\coach,\requesttype} \leq 1}_{\text{Assignment constraint}}
    % \right\}.
    \}.
\end{eqnarray}
The capacity constraints asserts that~$u_{\coach,\requesttype} = 0$ whenever~$\coach \notin \coaches_{\veccapacity,\timetogo}$, i.e., a request can only be assigned to coaches with sufficient residual capacity to accommodate it. The assignment constraint asserts that a request is assigned to at most one coach. We denote the optimal objective value by~\notation{\profit^*}{J(\veccapacity_1,1)}.

\subsubsection{Deterministic Linear Relaxation}
The state space of~\eqref{eq:dpformulation fixed price} is exponentially large, making it impractical. Therefore, we consider the following deterministic linear relaxation:
\begin{align}\label{model:fluid}
\tag{\textbf{Fluid}}
  \emphmath{\overline{\profit}} \coloneqq  \max\  &\sum\limits_{(\requesttype,\timetogo,\coach) \in \requesttypes \times \timehorizon \times \coaches} Pr[X \geq \timetogo] \cdot \price(\requesttype) \cdot  x_{\requesttype,\timetogo,\coach} &  \\
&
    \sum\limits_{\coach \in \coaches}       x_{\requesttype,\timetogo,\coach} \leq 
    \arrivalrate_{\requesttype}, 
    \forall  (\requesttype,\timetogo) \in \requesttypes \times \timehorizon  \label{eq:arrival upper bound}
    \\
&
    \sum_{\requesttype \in \requesttypes}
    \sum_{\timetogo \in \timehorizon}
        q_{\requesttype,\leg} x_{\requesttype,\timetogo,\coach}
    \leq
        \reservedcapacity,
    \forall (\leg,\coach) \in  \setLegs \times \coaches  \label{eq:capacity constraint} \\    
& 
    x_{\requesttype,\timetogo,\coach} \geq 0, 
    \forall (\requesttype,\timetogo,\coach) \in \requesttypes \times \timehorizon \times \coaches \nonumber
\end{align}
Formulation~\ref{model:fluid} uses continuous non-negative variables~$x_{\requesttype,\timetogo,\coach}$ to represent the assignment of a request of type~$\requesttype$ arriving in time step~$\timetogo$ to coach~$\coach$. Each term of the objective function weights the expected revenue collected from the assignment represented by~$x_{\requesttype,\timetogo,\coach}$ with the probability~$Pr[X \geq \timetogo]$ that the number of requests is larger than~$\timetogo$. Constraints~\eqref{eq:arrival upper bound} limit the sum of the assignments involving a request of type~$\requesttype$ arriving at time step~$\timetogo$ across all coaches to the arrival probability~$\arrivalrate_{\requesttype}$ of~$\requesttype$; this family of constraints can be interpreted as the continuous relaxation of the assignment constraints in~\eqref{feasible actions}. Lastly, constraints~\eqref{eq:capacity constraint} control the capacity utilization of each coach per leg and correspond to the capacity  constraints in~\eqref{feasible actions}. %We use~\emphmath{\overline{\profit}} to denote the optimal objective value of~\ref{model:fluid}.

\subsubsection{Upper Bound Guarantees}

We show that~$\overline{\profit} \geq \profit^*$, i.e.,  \ref{model:fluid}  provides an upper bound for the exact DP~\eqref{eq:dpformulation fixed price}. This result allows us to derive performance guarantees for a policy derived from~\ref{model:fluid}. The proof uses a similar choice of multipliers for a Lagrangian relaxation of~\eqref{eq:dpformulation fixed price} to~\citet{bai2023fluid}. 

\begin{proposition} 
The optimal value of~\ref{model:fluid} is an upper bound for the optimal value of~\eqref{eq:dpformulation fixed price}, i.e., $\overline{\profit} \geq \profit^*$.
\end{proposition}
\begin{proof}{Proof:}
    Let~$\mathcal{U}_-(\veccapacity)$ be a relaxation of~$\mathcal{U}(\veccapacity)$ that does not consider the capacity constraints on the legs but still preserves the assignment constraint (i.e., each request can be assigned to at most one coach); as~$\mathcal{U}(\veccapacity) \subseteq \mathcal{U}_-(\veccapacity)$, we obtain an upper bound for~\eqref{eq:dpformulation fixed price} by replacing~$\mathcal{U}(\veccapacity)$ with~$\mathcal{U}_-(\veccapacity)$.  We define a Lagrangian relaxation~\emphmath{J^{\boldsymbol{\alpha}}(\veccapacity,\timetogo)} of~\eqref{eq:dpformulation fixed price} over~$\mathcal{U}_-(\veccapacity)$ 
with multipliers~$\alpha_{\leg, \requesttype, \timetogo,\coach}$ associated with the capacity constraints as
\begin{eqnarray}\label{eq:lagrangian relaxation of DP}
    &&    \max_{u \in \mathcal{U}(\veccapacity)_-}
        % \biggl[
        \left[
            \sum_{\requesttype \in \requesttypes}\arrivalrate_{\requesttype}
            \left(
                \sum_{\coach \in \coaches}\price({\requesttype})u_{\coach,\requesttype }   +  \survival_{\timetogo} 
                J^{\boldsymbol{\alpha}}\left(\veccapacity -  \sum_{\coach \in \coaches}u_{\coach,\requesttype}\seatdemand_{\coach,  \requesttype},\timetogo+1\right)
            \right)
        +
        % \\
        % &&
    % \sum_{(\leg,\requesttype,\coach) \in \setLegs \times \requesttypes \times \coaches}
        \sum_{\requesttype \in \requesttypes}
        \sum_{\leg \in \setLegs}
        \sum_{\coach \in \coaches} 
            \alpha_{\leg, \requesttype, \timetogo,\coach} ( \reservedcapacity_{\coach,\leg} - 
            q_{\requesttype,\leg}  
    u_{\coach,\requesttype} )
        % \biggr]
        \right].
\end{eqnarray}
As in~\citet{bai2023fluid}, we study~$J^{\boldsymbol{\alpha}}(\veccapacity,\timetogo)$ when~$\alpha_{\leg, \requesttype, \timetogo,\coach} = 0$, $0 \leq \timetogo < |\timehorizon|$, and~$\alpha_{\leg, \requesttype, |\timehorizon|,\coach} = \arrivalrate_{\requesttype}  \eta_{\leg,\coach}$. The values of~$\eta_{\leg,\coach}$ are obtained through the solution of a linear program, from which we also derive~\ref{model:fluid}.

First, we show by induction in~$\timetogo$ that, for \notation{R_{\timetogo}}{\survival_\timetogo \survival_{\timetogo+1}\ldots\survival_{|\timehorizon|-1} 
=
\frac{Pr[X \geq |\timehorizon|]}{Pr[X \geq \timetogo]}
} and
\notation{R_{|\timehorizon|}}{1}, we have  
\begin{eqnarray}\label{eq:induction}
J^{\boldsymbol{\alpha}}(\veccapacity,\timetogo)
    &=&
        \sum_{k = \timetogo}^{\timehorizon}
        \sum_{\requesttype \in \requesttypes}
        \arrivalrate_{\requesttype}
        \max\left(
            0,
            \max_{\coach \in \coaches}
            \left(
            \frac{R_\timetogo}{R_k}\price({\requesttype})  
            -R_\timetogo 
            \sum_{\leg \in \setLegs}
                q_{\requesttype,\leg}\eta_{\leg,\coach}
            \right)^+
        \right)
         +
        R_\timetogo
        \sum_{(\leg,\coach) \in \setLegs \times \coaches}
            \eta_{\leg, \coach}  \reservedcapacity_{\coach,\leg}. 
            %\nonumber
\end{eqnarray}
We use~\notation{K_i}{\sum\limits_{k = \timetogo}^{\timehorizon}
        \sum\limits_{\requesttype \in \requesttypes}
        \arrivalrate_{\requesttype}
        \max\left(
            0,
            \max\limits_{\coach \in \coaches}
            \left(
            \frac{R_\timetogo}{R_k}\price({\requesttype})  
            -R_\timetogo 
            \sum\limits_{\leg \in \setLegs}
                q_{\requesttype,\leg}\eta_{\leg,\coach}
            \right)^+
        \right)}
        to represent the first term in \eqref{eq:induction}.

The base of the induction is $\timetogo = |\timehorizon|$, the last time step. As our analysis is restricted to the contribution of the requests arriving within~$\timehorizon$, we have $J^{\boldsymbol{\alpha}}(\veccapacity,|\timehorizon|+1) = 0$. Therefore, we can write~$J^{\boldsymbol{\alpha}}(\veccapacity,|\timehorizon|)$ for any~$\veccapacity$ as
\begin{eqnarray*}
J^{\boldsymbol{\alpha}}(\veccapacity,|\timehorizon|) 
        &=& 
        \max_{u \in \mathcal{U}_-(\veccapacity)}
        \left[
            \sum_{\requesttype \in \requesttypes}\arrivalrate_{\requesttype}
            \left(
                \sum_{\coach \in \coaches}\price({\requesttype})u_{\coach,\requesttype}  
            \right)
        +
        \sum_{\requesttype \in \requesttypes}
        \sum_{\coach \in \coaches} 
        \sum_{\leg \in \setLegs}
            \arrivalrate_{\requesttype}\eta_{\leg, \coach} ( \reservedcapacity_{\coach,\leg} - q_{\requesttype,\leg}  
    u_{\coach,\requesttype} )
        \right]
        \\
        &=& 
        \max_{u \in \mathcal{U}_-(\veccapacity)}
        \left[
            \sum_{\requesttype \in \requesttypes}
            \sum_{\coach \in\coaches}
            \arrivalrate_{\requesttype}
            u_{\coach,\requesttype}
            \price({\requesttype})  
        -
        \sum_{\requesttype \in \requesttypes}
        \sum_{\coach \in \coaches}
            \arrivalrate_{\requesttype}
            u_{\coach,\requesttype}
            \sum_{\leg \in \setLegs}
                q_{\requesttype,\leg}  
                \eta_{\leg, \coach}
        +
        \sum_{\requesttype \in \requesttypes}
        \sum_{\coach \in \coaches} 
        \sum_{\leg \in \setLegs}
            \arrivalrate_{\requesttype}   
            \eta_{\leg, \coach}  
            \reservedcapacity_{\coach,\leg}
        \right]
        \\
        &=& 
        \max_{u \in \mathcal{U}_-(\veccapacity)}
        \left[
            \sum_{\requesttype \in \requesttypes}
                \arrivalrate_{\requesttype} 
                \sum_{\coach \in \coaches}
                u_{\coach,\requesttype} 
                \left(\price({\requesttype})  
                -\sum_{\leg \in \setLegs} 
                q_{\requesttype,\leg} \eta_{\leg,\coach} \right) 
        \right]
        +
        \sum_{\requesttype \in \requesttypes}
\arrivalrate_{\requesttype}
        \sum_{\coach \in \coaches} 
        \sum_{\leg \in \setLegs}
            \eta_{\leg, \coach}  \reservedcapacity_{\coach,\leg}  
        \\
        &=& 
        \sum_{\requesttype \in \requesttypes}
                \arrivalrate_{\requesttype}
                \max\left(0,
                    \max_{\coach \in \coaches  }
                    \left(\price({\requesttype})  
                    -\sum_{\leg \in \setLegs} q_{\requesttype,\leg} \eta_{\leg,\coach} \right)^+
                    \right)
        +
        \sum_{\coach \in \coaches} 
        \sum_{\leg \in \setLegs}
            \eta_{\leg, \coach}  \reservedcapacity_{\coach,\leg}
            ,  
\end{eqnarray*}
where the last equality follows because~$\sum\limits_{\requesttype \in \requesttypes} \arrivalrate_{\requesttype} = 1$ (thus simplifying the second term) and exactly one assignment decision must be made for each request type, so~$\sum\limits_{\coach \in \coaches}u_{\coach,\requesttype}$ has at most one non-zero term. We obtain the values of~$K_{|\timehorizon|}$ and~$R_{|\timehorizon|}$ indicated in our induction hypothesis, so the base follows.

Next, from the induction hypothesis, we have 
% \begin{eqnarray*}
$
J^{\boldsymbol{\alpha}}(\veccapacity,\timetogo+1)
    = 
    K_{\timetogo+1} + R_{\timetogo+1}     \sum_{\leg \in \setLegs}
        \sum_{\coach \in \coaches} 
            \eta_{\leg, \coach} \reservedcapacity_{\coach,\leg}.
$
% \end{eqnarray*}
For~$J^{\boldsymbol{\alpha}}(\veccapacity,\timetogo)$ with~$\timetogo < |\timehorizon|$,  as the dual variables~$\alpha_{\leg, \requesttype, \timetogo,\coach}$ are all equal to zero, we have
\begin{eqnarray*}
    J^{\boldsymbol{\alpha}}(\veccapacity,\timetogo)
    &=& 
        \max_{u \in \mathcal{U}_-(\veccapacity)}
        \left(
            \sum_{\requesttype \in \requesttypes}\arrivalrate_{\requesttype}
            \left(
                \sum_{\coach \in \coaches}
                    \price({\requesttype})u_{\coach, \requesttype }   
                +  
                \survival_{\timetogo} J^{\boldsymbol{\alpha}}
                    \left(\veccapacity - \sum_{\coach \in \coaches}u_{\coach, \requesttype}\seatdemand_{\coach,  \requesttype} , \timetogo+1\right)
            \right)
        \right)
    \\
    &=&
    \max_{u \in \mathcal{U}_-(\veccapacity)}
        \left(
            \sum_{\requesttype \in \requesttypes}\arrivalrate_{\requesttype}
            \left(
                \sum_{\coach \in \coaches}\price({\requesttype})u_{\coach,\requesttype }   +  \survival_{\timetogo}K_{\timetogo+1} 
                + 
                \survival_{\timetogo}R_{\timetogo+1} \left( 
                    \sum_{\leg \in \setLegs}
                    \sum_{\coach \in \coaches} 
                        \eta_{\leg, \coach}  (\reservedcapacity_{\coach,\leg} - u_{\coach,\requesttype} q_{\requesttype,\leg} ) 
                \right)
            \right)
        \right) 
    \\
    &=&
    \max_{u \in \mathcal{U}_-(\veccapacity)}
            \sum_{\requesttype \in \requesttypes}
            \arrivalrate_{\requesttype}
        \left[ 
            \left(
                \sum_{\coach \in \coaches}
                    \price({\requesttype})
                     u_{\coach,\requesttype }
                - 
                \survival_{\timetogo}R_{\timetogo+1}  
                    \sum_{\leg \in \setLegs}
                    \sum_{\coach \in \coaches} 
                        \eta_{\leg, \coach} 
q_{\requesttype,\leg} u_{\coach,\requesttype } 
            \right)
        +
\left( 
        \survival_{\timetogo}K_{\timetogo+1} 
        +
        \survival_{\timetogo}R_{\timetogo+1}
                    \sum_{\leg \in \setLegs}
                    \sum_{\coach \in \coaches} 
                        \eta_{\leg, \coach}  \reservedcapacity_{\coach,\leg}   
\right)
\right]
    \\
    &=&
            \sum_{\requesttype \in \requesttypes}
            \arrivalrate_{\requesttype}
            \max\left( 
                0,
                \max_{\coach \in \coaches}
            \left(    
                        \price({\requesttype})   
                    - 
                    R_{\timetogo}  
                        \sum_{\leg \in \setLegs}
                            \eta_{\leg, \coach}  q_{\requesttype,\leg}
                \right)^+
                \right)
        +  
        \survival_{\timetogo}K_{\timetogo+1} 
        +
        R_{\timetogo}
                    \sum_{\leg \in \setLegs}
                    \sum_{\coach \in \coaches} 
                        \eta_{\leg, \coach}  \reservedcapacity_{\coach,\leg},   
\end{eqnarray*}
where the second equality follows from the induction hypothesis and the substitutions in the last equality hold because, by definition, $R_{\timetogo} = \survival_{\timetogo}R_{\timetogo+1}$. By expanding the first two terms, we obtain
\begin{eqnarray*}
\sum_{\requesttype \in \requesttypes}
            \arrivalrate_{\requesttype}
            \max\left(0,
                \max_{\coach \in \coaches}
            \left(    
                        \price({\requesttype})   
                    - 
                    R_{\timetogo}  
                        \sum_{\leg \in \setLegs}
                            \eta_{\leg, \coach}  q_{\requesttype,\leg}
                \right)^+
                \right)
    +       
        \survival_{\timetogo}
        \left[
            \sum_{k = \timetogo+1}^{\timehorizon}
            \sum_{\requesttype \in \requesttypes}
                \arrivalrate_{\requesttype}
                \max\left(
                    0,
                    \max_{\coach \in \coaches}
                        \left(
                        \frac{R_{\timetogo+1}}{R_k}\price({\requesttype})  
                        -R_{\timetogo+1} 
                        \sum_{\leg \in \setLegs}
                            q_{\requesttype,\leg}\eta_{\leg,\coach}
                    \right)^+
                \right)
        \right] 
    = 
    \\
\sum_{\requesttype \in \requesttypes}
            \arrivalrate_{\requesttype}
            \max\left(
                0,
                \max_{\coach \in \coaches}
            \left(    
                      \frac{R_{\timetogo}}{R_{\timetogo}}  \price({\requesttype})   
                    - 
                    R_{\timetogo}  
                        \sum_{\leg \in \setLegs}
                            \eta_{\leg, \coach} q_{\requesttype,\leg} 
                \right)^+
            \right)
+       
        \left[
        \sum_{k = \timetogo+1}^{\timehorizon}
        \sum_{\requesttype \in \requesttypes}
            \arrivalrate_{\requesttype}
            \max\left(
                0,
                \max_{\coach \in \coaches}
                \left(
                \frac{R_{\timetogo}}{R_k}\price({\requesttype})  
                -R_{\timetogo} 
                \sum_{\leg \in \setLegs}
                    q_{\requesttype,\leg}\eta_{\leg,\coach}
                \right)^+
            \right)
        \right]     
=
\\
\sum_{k = \timetogo}^{\timehorizon}
        \sum_{\requesttype \in \requesttypes}
            \arrivalrate_{\requesttype}
            \max\left(
                0,
                \max_{\coach \in \coaches}
                \left(
                \frac{R_{\timetogo}}{R_k}\price({\requesttype})  
                -R_{\timetogo} 
                \sum_{\leg \in \setLegs}
                    q_{\requesttype,\leg}\eta_{\leg,\coach}
                \right)^+
            \right)
            =
    K_{\timetogo}.
\end{eqnarray*}
Therefore, the induction hypothesis~\eqref{eq:induction} holds. From our assumption that~$Pr[X \geq 1] = 1$,  we have~$R_1 = Pr[X \geq |\timehorizon|]$ and, consequently, 
$\frac{R_1}{R_{\timetogo}}  = \frac{Pr[X \geq |\timehorizon|]}{\frac{Pr[X \geq |\timehorizon|]}{Pr[X \geq \timetogo]} } = 
Pr[X \geq \timetogo]
$. Therefore, the expected revenue is given by 
\begin{eqnarray*}
    J^{\boldsymbol{\alpha}}(\veccapacity_1,1) 
        &=& 
        \sum_{k = 1}^{\timehorizon}
        \sum_{\requesttype \in \requesttypes}
        \arrivalrate_{\requesttype}
        \max\left(
            0,
            \max_{\coach \in \coaches}
            \left(
            \frac{R_1}{R_k}\price({\requesttype})  
            -R_1 
            \sum_{\leg \in \setLegs}
                q_{\requesttype,\leg}\eta_{\leg,\coach}
            \right)^+
        \right)
         +
        R_1
        \sum_{\leg \in \setLegs}
        \sum_{\coach \in \coaches} 
            \eta_{\leg, \coach}  \reservedcapacity_{\coach,\leg} 
        \\
        &=&
        \sum_{k = 1}^{\timehorizon}
        \sum_{\requesttype \in \requesttypes}
        \arrivalrate_{\requesttype}
        \max\left(
            0,
            \max_{\coach \in \coaches}
            \left(
            Pr[X \geq k]
            \price({\requesttype})          
            -Pr[X \geq |\timehorizon|]
            \sum_{\leg \in \setLegs}
                q_{\requesttype,\leg}\eta_{\leg,\coach}
            \right)^+
        \right)
         +
        Pr[X \geq |\timehorizon|]
        \sum_{\leg \in \setLegs}
        \sum_{\coach \in \coaches} 
            \eta_{\leg, \coach}  \reservedcapacity_{\coach}.        
\end{eqnarray*}
The tightest Lagrangian relaxation associated with~$\boldsymbol{\alpha}$ is attained by the choice of values for the multipliers~$\eta_{\leg, \coach}$ that minimize~$J^{\boldsymbol{\alpha}}(\veccapacity,1)$. This is given by the following LP:
\begin{align}\label{model:lp}
\tag{\textbf{LP}}
    % &
    \min\  \sum_{\timetogo = 1}^{\timehorizon}
        \sum_{\requesttype \in \requesttypes}
        \arrivalrate_{\requesttype}
        {z}_{\timetogo,\requesttype}
         +
        Pr[X \geq |\timehorizon|]
        \sum_{\leg \in \setLegs}
        \sum_{\coach \in \coaches} 
            \eta_{\leg, \coach}  \reservedcapacity_{\coach} 
&
\\
    z_{\timetogo,\requesttype}  
    +  
    Pr[X \geq |\timehorizon|]\sum_{\leg \in \setLegs}
        q_{\requesttype,\leg}\eta_{\leg,\coach}   
    \geq 
    Pr[X \geq \timetogo]\price({\requesttype}),          
    & 
    \forall (\requesttype,\timetogo,\coach) \in \requesttypes \times \timehorizon \times  \coaches
    \\
    \eta_{\leg, \coach} \geq 0, 
    &
    \forall (\leg,\coach) \in \setLegs \times \coaches \nonumber \\
    {z}_{\timetogo,\requesttype} \geq 0,
    &
    \forall (\timetogo,\requesttype) \in \timehorizon \times \requesttypes \nonumber
\end{align}
\begin{comment}
One of the family of constraints composing the dual formulation, defined for each~$(\leg,\coach) \in  \setLegs \times \coaches$,  reads as follows:
\[
 Pr[X \geq |\timehorizon|]
    \sum_{\timetogo \in \timehorizon}
    \sum_{\requesttype \in \requesttypes}
    % \sum_{\leg \in \setLegs}
    q_{\requesttype,\leg} x_{\requesttype,\timetogo,\coach}
    \leq
    Pr[X \geq |\timehorizon|]
        % \sum_{\leg \in \setLegs}
        % \sum_{\coach \in \coaches} 
             \reservedcapacity_{\coach}.
    % & 
    % \forall (\leg,\coach) \in  \setLegs \times \coaches 
\]
\end{comment}
Finally, for~$Pr[X \geq |\timehorizon|] > 0$, we obtain a dual formulation of~\ref{model:lp} that is  equivalent to~\ref{model:fluid}. Therefore, it follows from strong duality that~$\overline{\profit} =  J^{\boldsymbol{\alpha}}(\veccapacity,1)$, i.e., \ref{model:fluid} provides an upper bound for the exact DP~\eqref{eq:dpformulation fixed price}. %$\hfill \blacksquare$ 
\end{proof}

\subsubsection{Asymptotically Analysis of Policy Based on Fluid Model}

We investigate the asymptotic performance of the $\arrivalrate$-Policy, which is based on an optimal solution to~\ref{model:fluid}. The $\arrivalrate$-Policy makes resources for request type~$\requesttype$ available in time step~$\timetogo$ with probability~$\sum\limits_{\coach \in \coaches} \frac{x_{\requesttype,\timetogo,\coach}}{\arrivalrate_{\requesttype}}$ whenever the current residual capacity~$\veccapacity_\timetogo$ is sufficiently large to accommodate a request of this type. Each accepted request is arbitrarily assigned to a coach with sufficient residual capacity. Theorem~\ref{thm: fluid optimality} characterizes the performance of $\arrivalrate$-Policy when~$\delta \geq \frac{2}{\reservedcapacity}$. 

\begin{theorem}\label{thm: fluid optimality}
For~$\delta \geq \frac{2}{\reservedcapacity}$,
the expected revenue~$E[\profit]$ attained by the policy is 
\begin{eqnarray}\label{eq: fluid approximation}
    E[\profit]  
\geq
    \left[
    1 - 
|\setLegs|
\exp{\left(  
    -\frac{  
    |\coaches|
            \left(
                \frac{1}{\delta} 
    -\reservedcapacity
            \right)^2 
        }    
        {
            2\reservedcapacity
            + 
            \frac{2}{3} 
                \left(
                   \frac{1}{\delta} 
    -\reservedcapacity
                \right) 
            }  
\right)}
\right]\overline{\profit},     
\end{eqnarray}
where~$\overline{\profit}$ is the optimal objective value of~\ref{model:fluid}.
\end{theorem}
\begin{proof}{Proof:} Let~$\vecx$ be an optimal solution to~\ref{model:fluid}. Request~$\request_{\timetogo}$ arrives with probability~$Pr[X \geq \timetogo]$, and from our Markovian assumption on the arrival rates, $\request_\timetogo$ is of type~$\requesttype$ with probability~$\arrivalrate_{\requesttype}$. As request type~$\requesttype$ is offered in time step~$\timetogo$ with probability~$\sum\limits_{\coach \in \coaches}  \frac{x_{\requesttype,\timetogo,\coach}}{\arrivalrate_{\requesttype}}$, the policy tries to assign a request of type~$\requesttype$ in time step~$\timetogo$ with probability~$\sum\limits_{\coach \in \coaches}x_{\requesttype,\timetogo,\coach}$. Therefore, by using a Bernoulli variable~$G_{\requesttype,\timetogo}$ to denote the event where the residual capacity is sufficiently large to accommodate a request of type~$\requesttype$ in time step~$\timetogo$, we can write the expected revenue~$E[\profit]$ attained by the policy as 
\begin{eqnarray*}
E[\profit] = 
        \sum_{\timetogo \in \timehorizon}
        \sum_{\requesttype \in \requesttypes}
        \price(\requesttype)
        \underbrace{Pr[G_{\requesttype,\timetogo} = 1]}_{\text{Sufficient residual capacity to assign~$\requesttype$}}
        \underbrace{
        Pr[X \geq \timetogo]
        }_{\text{There are at least $\timetogo$ requests.}}
        \underbrace{
        \sum_{\coach \in \coaches}
             x_{\requesttype,\timetogo,\coach}.
            }_{\text{Policy tries to assign a request of type~$\requesttype$.}}       
\end{eqnarray*}
As in~\S\ref{sec:random order model}, we aggregate resource consumption to circumvent technical difficulties with coaches and  analyze seat availability in each leg. Namely, a request of type~$\requesttype$ can be serviced if, for every leg~$\leg$ in~$\setLegs(\requesttype)$, the total number of seats in~$\leg$ that have been already assigned is at most~$|\coaches|\reservedcapacity (1 - \delta)$ across all coaches; recall that~$\reservedcapacity$ is the original seating capacity of each coach. In particular, a lower bound on the number of accepted requests containing leg~$\leg$ in their itinerary is~$
\frac{|\coaches|\reservedcapacity}{\delta \reservedcapacity} 
= 
\emphmath{\frac{|\coaches|}{\delta} } 
$.

We identify a lower bound for~$Pr[G_{\requesttype,\timetogo} = 1]$ by estimating the occupancy of each leg~$\leg$ throughout the entire selling horizon. We define a Bernoulli variable~$N_{\leg,\timetogo}$ indicating whether there is demand for seats in leg~$\leg$ in step~$\timetogo$. The probability of this event is given by the sum of the probabilities with which a request of type~$\requesttype$ such that~$\setLegs(\requesttype)$ contains~$\leg$ appears in time step~$\timetogo$. For each~$\requesttype$, this probability is the product of its arrival rate~$\arrivalrate_{\requesttype}$ with its offering probability~$\sum\limits_{\coach \in \coaches}\frac{x_{\requesttype,\timetogo,\coach}}{\arrivalrate_{\requesttype}}$, so we have 
\begin{eqnarray*}
Var\left[N_{\leg,\timetogo}\right]
\leq
    E\left[N_{\leg,\timetogo}\right]
=
    Pr[N_{\leg,\timetogo}] 
&=& 
    \sum_{\requesttype \in \requesttypes: \leg \in \setLegs(\requesttype)}
    \sum_{\coach \in \coaches}
        \frac{x_{\requesttype,\timetogo,\coach}}{\arrivalrate_{\requesttype}}
        \arrivalrate_{\requesttype}    
       =
    \sum_{\requesttype \in \requesttypes: \leg \in \setLegs(\requesttype)}
    \sum_{\coach \in \coaches}
    x_{\requesttype,\timetogo,\coach}.
\end{eqnarray*}
We are interested in the number of requests involving~$\leg$ across all time steps. For this, we have
\[
\sum_{\timetogo \in \timehorizon}Var\left[N_{\leg,\timetogo}\right]
\leq
    \sum_{\timetogo \in \timehorizon}
        E\left[N_{\leg,\timetogo}\right]
=
    \sum_{\timetogo \in \timehorizon}
    \sum_{\requesttype \in \requesttypes: \leg \in \setLegs(\requesttype)}
    \sum_{\coach \in \coaches}
      x_{\requesttype,\timetogo,\coach}
\leq
    \sum_{\timetogo \in \timehorizon}
    \sum_{\requesttype \in \requesttypes: \leg \in \setLegs(\requesttype)}
    \sum_{\coach \in \coaches}
        q_{\requesttype,\leg} x_{\requesttype,\timetogo,\coach}
\leq
    |\coaches|\reservedcapacity,
\]
where the last two inequalities follow because~$q_{\requesttype,\leg} \geq 1$ if~$\leg \in \setLegs(\requesttype)$ and~$\vecx$ is a feasible solution to~\ref{model:fluid}. Note that the bounds above do not depend on~$|\timehorizon|$, but only on $|\coaches|$ and~$\reservedcapacity$, which are parameters of the problem. Moreover, the bound is trivial, as there are $|\coaches|\reservedcapacity$ seats available in a train, so no more than~$|\coaches|\reservedcapacity$ requests containing leg~$\leg$ can be assigned. Nevertheless, this bound is sufficient for our analysis. Next, we identify an upper bound for the probability with which the number of assigned requests surpasses~$\frac{|\coaches|}{\delta}$:
\begin{eqnarray*}
Pr\left[
    \sum_{\timetogo \in \timehorizon}
        N_{\leg,\timetogo} 
\geq
    \frac{|\coaches|}{\delta}
\right]
&\leq&
Pr\left[
    \sum_{\timetogo \in \timehorizon}
        (N_{\leg,\timetogo} - E[N_{\leg,\timetogo}] ) 
\geq 
 |\coaches|\left( 
    \frac{1}{\delta} 
    -\reservedcapacity
    \right)
\right]
% \\
% &\leq&
\leq
\exp{\left(  
    -\frac{  
        \frac{|\coaches|^2}{2} 
            \left(
                \frac{1}{\delta} 
    -\reservedcapacity
            \right)^2 
        }    
        {
            \sum\limits_{\timetogo \in \timehorizon}
                E[\left(N_{\leg,\timetogo} - E[N_{\leg,\timetogo}] \right)^2] 
            + 
            \frac{|\coaches|}{3} 
                \left(
                   \frac{1}{\delta} 
    -\reservedcapacity
                \right) 
            }  
\right)}
\\
&=&
\exp{\left(  
    -\frac{  
        \frac{|\coaches|^2}{2} 
            \left(
                \frac{1}{\delta} 
    -\reservedcapacity
            \right)^2 
        }    
        {
            \sum\limits_{\timetogo \in \timehorizon}
                Var[N_{\leg,\timetogo}] 
            + 
            \frac{|\coaches|}{3} 
                \left(
                   \frac{1}{\delta} 
    -\reservedcapacity
                \right) 
            }  
\right)}
% \\
% &\leq&
\leq
\exp{\left(  
    -\frac{  
        \frac{|\coaches|^2}{2} 
            \left(
                \frac{1}{\delta} 
    -\reservedcapacity
            \right)^2 
        }    
        {
            |\coaches|\reservedcapacity
            + 
            \frac{|\coaches|}{3} 
                \left(
                   \frac{1}{\delta} 
    -\reservedcapacity
                \right) 
            }  
\right)}
% =
\\
&=&
\exp{\left(  
    -\frac{  
        |\coaches|
        % \frac{|\coaches|}{2} 
            \left(
                \frac{1}{\delta} 
    -\reservedcapacity
            \right)^2 
        }    
        {
            2\reservedcapacity
            + 
            \frac{2}{3} 
                \left(
                   \frac{1}{\delta} 
    -\reservedcapacity
                \right) 
            }  
\right).}
\end{eqnarray*}
% https://en.wikipedia.org/wiki/Bernstein_inequalities_(probability_theory)
The first inequality holds because $E\left[\sum\limits_{\timetogo \in \timehorizon}N_{\leg,\timetogo}\right] \leq  \reservedcapacity|\coaches|$. The second inequality follows from the one-sided Bernstein inequality (see, e.g., \cite{cesa2006prediction}); the conditions of the inequality hold, as the random variables~$N_{\leg,\timetogo} - E[N_{\leg,\timetogo}]$ are pairwise independent, have expectation equal to zero, and are bounded by 1. The third inequality  follows because $Var\left[\sum\limits_{\timetogo \in \timehorizon}N_{\leg,\timetogo}\right] \leq |\coaches|\reservedcapacity$. Finally, we cannot have~$\delta = \frac{1}{4\reservedcapacity}$ (as the denominator equals zero in this case), but as~$\delta \geq \frac{1}{k}$, the expression is well-defined.

The lower bound above holds for a single leg~$\leg$, so we apply union bound to derive a bound for~$ Pr[G_{\requesttype,\timetogo} = 1]$, which encompasses all~$|\setLegs|$ legs in the worst case:
\begin{eqnarray}
    Pr[G_{\requesttype,\timetogo} = 1] 
&\geq&
    Pr\left[
        \sum_{\requesttype \in \requesttypes}
            N_{\leg,\timetogo}
        <
        \frac{|\coaches| }{\delta},
        \forall \leg \in \setLegs(\requesttype)
    \right]
=
    1 - Pr\left[  \exists \leg \in \setLegs(\requesttype), 
        \sum_{\requesttype \in \requesttypes}
            N_{\leg,\timetogo}
        \geq
        \frac{|\coaches| }{\delta}
    \right] \nonumber
\\
&\geq&
1 - 
|\setLegs|
\exp{\left(  
    -\frac{  
    |\coaches|
        % \frac{|\coaches|}{2} 
            \left(
                \frac{1}{\delta} 
    -\reservedcapacity
            \right)^2 
        }    
        {
            2\reservedcapacity
            + 
            \frac{2}{3} 
                \left(
                   \frac{1}{\delta} 
    -\reservedcapacity
                \right) 
            }  
\right)} \label{eq:G lower bound}
\end{eqnarray}
The lower bound in~\eqref{eq:G lower bound} holds for all values of~$\requesttype$ and~$\timetogo$. Therefore, we obtain the following bound for the expected objective value~$E[\profit]$ attained by the policy in terms of the optimal objective value~$\overline{\profit}$ of~\ref{model:fluid}: 
\begin{eqnarray*}
    E[\profit]  
&=&             Pr[G_{\requesttype,\timetogo} = 1]
    \sum_{\timetogo \in \timehorizon}
    \sum_{\requesttype \in \requesttypes}
    \sum_{\coach \in \coaches}
        Pr[X \geq \timetogo]
        \price(\requesttype)     
        x_{\requesttype,\timetogo,\coach}
\geq 
    Pr[G_{\requesttype,\timetogo} = 1]\overline{\profit}
% \\
% &\geq&
    % Pr[G_{\requesttype,\timetogo} = 1]\overline{\profit}
% \\
% &\geq&
\geq
    \left[
    1 - 
|\setLegs|
\exp{\left(  
    -\frac{  
    |\coaches|
        % \frac{|\coaches|}{2} 
            \left(
                \frac{1}{\delta} 
    -\reservedcapacity
            \right)^2 
        }    
        {
            2\reservedcapacity
            + 
            \frac{2}{3} 
                \left(
                   \frac{1}{\delta} 
    -\reservedcapacity
                \right) 
            }  
\right)}
\right]\overline{\profit}.   % \hfill \blacksquare 
\end{eqnarray*}
% where~$J^d$ denotes the optimal objective value of~\ref{model:fluid}. 
%     $\hfill \blacksquare$
\end{proof}

By replacing into~\eqref{eq: fluid approximation} the values~$|\coaches| = 20$, $\reservedcapacity = 100$, and~$\delta = 0.06$, as observed in practice, we obtain
\begin{eqnarray*}
    E[\profit]   
&\geq&
    \left[
    1 - 
|\setLegs|
\exp{\left(  
    -\frac{  
    |\coaches|
            \left(
                \frac{1}{\delta} 
    -\reservedcapacity
            \right)^2 
        }    
        {
            2\reservedcapacity
            + 
            \frac{2}{3} 
                \left(
                   \frac{1}{\delta} 
    -\reservedcapacity
                \right) 
            }  
\right)}
\right]\overline{\profit}
% \\
% &\geq&
=
    \left[
    1 
    - 
    4 \exp{\left(  
    -\frac{  
        10 
            \left(
                16.667 - 100
            \right)^2 
        }    
        {
            66.667
            +
            5.555
        }  
\right)} \right]\overline{\profit} 
\approx \overline{\profit},
\end{eqnarray*}
i.e., $\arrivalrate$-Policy is arbitrarily asymptotically optimal. Finally, the value of~$\delta$ plays an important role in our analysis. Namely, this policy is optimal for every~$\delta \geq \frac{2}{\reservedcapacity}$, i.e., whenever we consider groups of size at least two. Interestingly, our analysis does not yield the same guarantee in scenarios where each request consists of a single person, i.e.,~$\delta = \frac{1}{\reservedcapacity}$. In this case, an optimal policy offers a request of type~$\requesttype$ with probability~$\theta \sum\limits_{\coach \in \coaches} \frac{x_{\requesttype,\timetogo,\coach}}{\arrivalrate_{\requesttype}}$ for a carefully chosen value of~$\theta$, as in~\citet{bai2023fluid}, and the competitive ratio is slightly larger than 0.91. We present the proof of Corollary~\ref{cor: fluid optimality group 1} in the Appendix. 
\begin{corollary}\label{cor: fluid optimality group 1}
    For~$\delta = \frac{1}{\reservedcapacity}$, the expected revenue~$E[\profit]$ attained by the policy is 
\begin{eqnarray*}
    E[\profit]  
\geq
\theta \left(1 - |\setLegs| 
 \exp{\left(  
    -\frac{  
        \frac{3|\coaches|}{2\delta} 
            \left(
                1
                -
                \frac{\delta}{|\coaches|}
                -    
                \theta\reservedcapacity\delta
            \right)^2 
        }    
        {
            1
            -
            \frac{\delta}{|\coaches|}
            +    
            2\theta\reservedcapacity\delta
        }  
\right)} \right) \overline{\profit}.
\end{eqnarray*}
In particular, for~$|\coaches| = 20$, $\reservedcapacity = 100$, and~$\delta = 0.06$, the maximum factor is~$0.916$, achieved when~$\theta \approx 0.9218$. 
\end{corollary}

\subsubsection{Fixed Assignment Policy}

With information about the arrival rates, we can compute the probabilities with which the~$j^{th}$ request of type~$\requesttype$ will arrive within the selling horizon. With that, we can  create \textit{pre-assignment plans}, in which seats are assigned to requests a priori, before the beginning of the selling horizon. We present an algorithm that computes pre-assignment plans using the arrival probabilities of each request to scale the revenue collected from each assignment. Let~\emphmath{X_\requesttype} be the random variable denoting the number of requests of type~$\requesttype$ that arrive within the selling horizon. The support of each~$X_\requesttype$ is~$\mathbb{N}^+$, so there is no upper limit on the number of requests of any type. Therefore, for computational tractability, we restrict the set of requests to~\notation{\aprioriassignments_\delta}{ \{(\requesttype,j) \in \requesttypes \times \mathbb{N}:  Pr[X_{\requesttype} \geq j] \geq \delta\}}, i.e., we only consider requests with arrival probabilities larger than some given~$\delta > 0$. We also use the arrival probabilities to scale the revenue collected from the assignments. This approach consists of solving~\ref{model:primal}, the MIP formulation presented below:
\begin{align}\label{model:primal}
\tag{\textbf{A-Priori}($\delta$)}
    % &
    \max\  &\sum\limits_{(\requesttype,j,\coach) \in \aprioriassignments_\delta \times \coaches} Pr[X_{\requesttype} \geq j] \cdot \price(\requesttype) \cdot  x_{\requesttype,j,\coach} &  \\
%(\alpha) 
&
    \sum\limits_{\coach \in \coaches}  x_{\requesttype,j,\coach} \leq 1, 
     & 
    \forall \request = (\requesttype,j)  \in \aprioriassignments_\delta
    \nonumber
\\
%(\beta) 
&
    \sum\limits_{\substack{(\requesttype,j,\coach) \in \aprioriassignments_\delta \times \coaches \\ 
    \leg \in \setLegs(\requesttype)}} \numPass(\requesttype) \cdot x_{\requesttype,j,\coach} \leq \reservedcapacity,
     & 
    \forall (\leg,\coach) \in  \setLegs \times \coaches 
    \nonumber    
\\    
%(\gamma) 
& 
    \sum\limits_{\coach \in \coaches}  x_{\requesttype,j+1,\coach} \leq
\sum\limits_{\coach \in \coaches}  x_{\requesttype,j,\coach}, 
    &
    \forall (\requesttype,j) \in \aprioriassignments_\delta 
    \label{eq: early requests first}
    \\
& 
    x_{\requesttype,j,\coach} \in \{0,1\}, 
    &
    \forall (\requesttype,j,\coach) \in \aprioriassignments_\delta \times \coaches \nonumber
\end{align}
The decision variables~$x_{\requesttype,j,\coach}$ and the first two families of inequalities are identical to their counterparts in~\ref{model:offline}. Inequalities~\eqref{eq: early requests first} force~\ref{model:primal} to assign the~$j$-th request of type~$\requesttype$ only if all the previous requests of the same type have been assigned as well; inequalities~\eqref{eq: early requests first} are valid because all requests of the same type yield the same revenue and~$Pr[X_{\requesttype} \geq j] \geq 
Pr[X_{\requesttype} \geq j+1]$. We use~\emphmath{\mathcal{A}_\delta} and~\emphmath{\requests_\delta} to denote the set of  assignments and serviced requests associated with a pre-assignment plan identified by~\ref{model:primal}.

The main drawback of adopting a fixed assignment policy is the uncertainty associated with the arrivals; some unlikely yet valuable requests may  appear, whereas requests in~$\requests_\delta$ may never arrive. Moreover, one can construct pathological scenarios for which this policy is not~$c$-competitive for any constant~$c$. To see this, note that a sufficiently large number of request types with low arrival rates (and, consequently, with arrival probability inferior to~$\delta$) and extremely high prices would always compose an optimal solution while being discarded by the policy. On the positive side, such worst-case scenarios are not observed in practice, and the fixed assignment allows the railway company to increase the integration of the assignment decisions. %Our numerical studies show that this policy is effective in practical scenarios (see~\S\ref{sec:experiments}).

\section{\Problem with Fairness Constraints}

This section investigates variants of~\Problem with fairness constraints. For~\Problem-FCFS we present an exact MIP formulation for the offline case and propose an algorithm for the online case with arrival rates. We conclude with a technical discussion  of~\Problem-SFCFS.

\subsection{The Offline~\Problem-FCFS}\label{sec:The offline CRP-FCFS}

In the offline~\Problem-FCFS, the entire arrival sequence~$\requests$ is known. We focus on structural properties with direct algorithmic implications, i.e., our goal is to develop efficient mathematical programming algorithms by extending formulation~\ref{model:offline} to enforce fairness.

\subsubsection{Forbidding Unfair Rejections} \label{section: minimum fairness conditions}

According to the FCFS constraint, an incoming request~$\request_\timetogo$ cannot be rejected if the current residual seating capacity~$\veccapacity_\timetogo$ is large enough to accommodate~$\request_\timetogo$ in at least one coach, i.e., if $\coaches_{\veccapacity_\timetogo,\request_\timetogo} \neq \emptyset$.
We model this condition by employing binary \textit{fairness} variables~\emphmath{z_{\request_\timetogo,\leg, \coach}}, which indicate that request~$\request_\timetogo$ cannot be assigned to coach~$\coach$ due to lack of residual capacity (i.e., fewer than~$\numPass(\request_\timetogo)$ available seats) in leg~$\leg \in \setLegs(\request_\timetogo)$. For every request~$\request_\timetogo$ and coach~$\coach$, a fairness variable~$z_{\request_\timetogo,\leg, \coach}$ must be activated if and only if~$\request_\timetogo$ is rejected, i.e., if $y_{\request_\timetogo} = 0$ in formulation~\ref{model:offline}. This condition is modeled as follows:
\begin{eqnarray}\label{eq:Unfair Activation}
y_{\request_\timetogo} + \sum_{\leg \in \setLegs(\request_\timetogo)} z_{\request_\timetogo,\leg,\coach} = 1 ,  \forall \request_\timetogo \in \requests, \coach \in \coaches.
\end{eqnarray}
Constraint~\eqref{eq:Unfair Activation} activates at most one fairness variable per coach and request. Note that two or more legs may have insufficient capacity to accommodate an incoming request, but our model just needs to activate a single fairness variable per coach to enforce FCFS. 

We use each fairness variable~$z_{\request_\timetogo,\leg,\coach}$ to impose a occupancy levels that would justify rejections under FCFS. Namely, for each leg~$\leg$ coach~$\coach$, and request~$\request_{\timetogo}$ we have
\begin{eqnarray}\label{eq:Unfair Rejections}
\centering
        \sum_{\request' \in \requests_{\timetogo-1}: 
        \leg \in \setLegs(\request') 
        % \arrivalorder_{\request'} < \arrivalorder_{\request} 
        } 
            \numPass(\request') x_{\request',\coach} 
    \geq  
        (\reservedcapacity -   \numPass(\request_\timetogo)+1)z_{\request_\timetogo,\leg,\coach}.
\end{eqnarray}
If~$z_{\request_\timetogo,\leg,\coach}$ is activated, inequality~\eqref{eq:Unfair Rejections} imposes the assignment of a set of requests in~$\requests_{\timetogo-1}$ to coach~$\coach$ that occupy at least~$\reservedcapacity -   \numPass(\request_\timetogo)+1$ seats in leg~$\leg$; under these circumstances, $\request_\timetogo$ cannot be assigned to~$\coach$.

The incorporation of the families of constraints~\eqref{eq:Unfair Activation} and~\eqref{eq:Unfair Rejections} and fairness variables~$\boldsymbol{z}$ into~\ref{model:offline} yields formulation~\ref{model:deterministic}, which solves~\Problem-FCFS (see Section~\ref{sec:Deterministic Formulation} in the electronic companion). Unfortunately, preliminary experiments show that~\ref{model:deterministic} is computationally challenging. Therefore, we investigate other techniques and families of cuts to enhance the computational performance of our algorithm.

\subsubsection{Forward Filtering}\label{sec: forward filtering}
The family of inequalities~\eqref{eq:Unfair Rejections} forces the assignment of requests preceding some rejected request~$\request_\timetogo$. A complementary idea consists of using the fairness variables to limit the number of seats for legs with exhausted capacity in~$\setLegs(\request_\timetogo)$ assigned \textit{after} the rejection of~$\request_\timetogo$ to~$\numPass(\request_\timetogo) - 1$, i.e., if~$\request_\timetogo$ is rejected, we cannot assign the seats it would occupy to future requests. This condition is modeled as follows:
\begin{eqnarray}\label{eq:Unfair Rejections B}
\centering
        \sum_{\request' \in \requests \setminus \requests_{\timetogo}: 
        \leg \in \setLegs(\request'), 
        } \numPass(\request') x_{\request',\coach} 
    \leq 
        (1 - z_{\request_\timetogo,\leg,\coach}) \reservedcapacity  + 
        (\numPass(\request) - 1)z_{\request_\timetogo,\leg,\coach}.
\end{eqnarray}

We explore a similar idea to define the concept of  \textit{dominating requests}. Given a pair of requests~$\request$ and~$\request'$, we say that~\textit{$\request'$ dominates~$\request$} if~$\request <\request'$ ($\request'$ arrives after~$\request$), $\setLegs(\request) \subseteq \setLegs(\request')$ (the itinerary of~$\request'$ contains all legs in~$\setLegs(\request)$), and~$\numPass(\request) \leq \numPass(\request')$ ($\request'$ does not have fewer people than~$\request$); we use~$\request' \succ \request$ (or~$\request \prec \request'$) to indicate that~$\request'$ dominates~$\request$. Let~\notation{\requests_{\request}^{\succ}}{ \{ \request' \in \requests: \request \succ \request' \}  } be the set of requests dominated by~$\request$. Remark~\ref{prop:Dominating requests} shows the connection between dominance and FCFS.

\begin{remark}\label{prop:Dominating requests} In feasible solutions to~\Problem-FCFS, the assignment of~$\request$ implies the assignment of~$\requests_{\request}^{\succ}$, i.e., 
\begin{eqnarray}\label{eq:Unfair Assignments}
\centering
\sum_{\request' \in \requests_{\request}^{\succ}} y_{\request'} \geq |\requests_{\request}^{\succ}|y_{\request}, \forall \request \in \requests.
\end{eqnarray}
\end{remark}
Inequalities~\eqref{eq:Unfair Rejections B} and~\eqref{eq:Unfair Assignments} are insufficient to enforce the FCFS constraint. In particular, they do not capture scenarios where a request associated with a large group of people is rejected, whereas requests arriving later from smaller groups demanding the same resources are assigned; Example~\ref{example: incomplete constraint} illustrates this situation. 
%Such cases are observed mainly when the residual capacity is small (i.e., the train is almost full).  
\begin{example}%[Constraints~\eqref{eq:Unfair Rejections B} do not enforce fairness]
\label{example: incomplete constraint} Let~$\request$ and~$\request'$
be requests such that~$\numPass(\request) = 4$, $\numPass(\request') = 3$, $\setLegs(\request) = \setLegs(\request')$, $\request < \request'$, and~$z_{\request,\leg,\coach} = 1$ for some~$\leg$ in~$\setLegs(\request)$.  Inequalities~\eqref{eq:Unfair Rejections B} and~\eqref{eq:Unfair Assignments} do not preclude 
the rejection of~$\request$ followed by the assignment of~$\request'$ in scenarios where a coach with residual capacity at least~$\numPass(\request)$ was available upon the arrival of~$\request$.
\end{example}

\subsection{The Online \Problem-FCFS}

\newcommand{\block}{\ensuremath{b}}
\newcommand{\blocks}{\ensuremath{B}}

Algorithm~\ref{alg:crp-fcfs} describes the~\texttt{FCFS Policy}, a procedure tailored for the online~\Problem-FCFS based on~\ref{model:primal}. \Problem-FCFS does not admit unilateral rejections, so the policy solves~\ref{model:primal} periodically and adjusts the pre-assignment plan in real-time to accommodate incoming requests that have not been pre-assigned but can nevertheless be serviced.  

\begin{algorithm}[H]
\begin{algorithmic}[1]
\footnotesize
\State set~$\vecx = \boldsymbol{0}, \veccapacity = \veccapacity_1$;
\For{each $\block$ in $\blocks$}
\State Obtain~$\mathcal{A}_\delta$ and~$\requests_\delta$ by solving~\ref{model:primal} with the arrival probabilities of block~$\block$ 
\For{each $\timetogo$ in $\timehorizon_\block$}
\If{$\request_\timetogo \in \requests_\delta$}
\State set $\vecx_{\request_\timetogo,\coach} = 1$ and $\veccapacity = \veccapacity - \seatdemand^{\coach,  \request_{\timetogo}}$ for~$(\request_\timetogo,\coach) \in \mathcal{A}_\delta$
% \Else{}
\ElsIf{$\coaches_{\veccapacity,\request_\timetogo} \neq \emptyset$}
% \State reject~$\request_\timetogo$
% \Else{}
\State Obtain~$\mathcal{A}_\delta$ and~$\requests_\delta$ such that~$\request_\timetogo \in \requests_\delta$ by solving~\ref{model:primal} with the arrival probabilities of~$\block$ 
\EndIf{}    
% \EndIf{}
\EndFor{}
\EndFor{}
\end{algorithmic}
\caption{\texttt{FCFS Policy}}
\label{alg:crp-fcfs} 
\end{algorithm}

Algorithm~\ref{alg:crp-fcfs}  uses the pre-assignment plan for its assignment decisions. If a pre-assigned request arrives, \texttt{FCFS Policy} adopts the choice of the current plan. Otherwise, it solves~\ref{model:primal} again with the additional constraint that the current request must be assigned if the residual capacity is large enough to accommodate it; otherwise, the request is rejected. 
% \paragraph{Periodic Updates}
%In practice, the 
The arrival probabilities are constantly changing, but 
%so ideally, one would solve~\ref{model:primal} upon the arrival of each request.  However, reject-or-assign decisions must be made in real-time, and 
formulation~\ref{model:primal} is computationally challenging for real-time decision-making for real instances. Therefore, %to avoid frequent changes that require significant computational efforts, 
Algorithm~\ref{alg:crp-fcfs} partitions the set of time steps into~$\block$ \textit{blocks}~$\timehorizon_1, \timehorizon_2,\ldots,\timehorizon_\block$ (e.g., each representing a day).
%; we refer to each of these subsets as \textit{blocks}. 
The arrival probabilities are updated only in the first time step of each block (e.g., at the beginning of the day), so the instances of~\ref{model:primal} solved within a block are relatively similar. 
%In practice, a block would be associated with a day, and a new solution with updated arrival probabilities would be computed at the beginning of the day.

% \paragraph{Diversification of Assignments }
% The pre-assignment plans produced by~\ref{model:primal} may still ignore requests that are very likely to occur. This is a problem for the adoption of~\ref{model:primal}, as the pre-assignment plan may quickly become inadequate in the presence of FCFS constraints.

% We mitigate this issue by setting a lower bound on the number of requests that must be serviced by~\ref{model:primal} for each request type. We use the following family of constraints for each request type~$\requesttype$:
% \begin{eqnarray*}
% \sum\limits_{\substack{(\requesttype,j,\coach) \in \aprioriassignments_\delta \times \coaches \\ 
%     1 \leq j \leq k_{\requesttype}}} x_{\requesttype,j,\coach} \geq 
%      \min\left(5,\argmax\limits_{j \in \mathbb{N}} \left(Pr[X_{\requesttype} \geq j] \geq 1-\frac{\delta}{10}\right)  \right).
% \end{eqnarray*}
% We limit the bound to five to avoid an overwhelming preference for requests with groups of size one, which are far more frequent in practice. This is also motivated because unassigned large requests are more disruptive to pre-assignment plans. 

% \paragraph{Solve Offline \Problem-FCFS}
% Create an arrival order based on the arrival probabilities and solve the offline problem. In particular, diversify assignments by fixing the assignments up to a certain point (e.g., until first possible rejection happens). 

\subsection{\Problem with Strict FCFS Constraints and Delayed Assignments}\label{sec:temporal fairness}

\Problem-SFCFS is a variant of the problem that provides more assignment flexibility in exchange for a stricter fairness condition. In~\Problem-SFCFS, the company only needs to maintain a set~\emphmath{\overline{\requests}_{\timetogo}} $\subseteq \requests_\timetogo$ of \textit{accepted} requests in each type step~$\timetogo$, rather than a set of fixed and irrevocable assignments. The final assignment of accepted requests to coaches must only happen after the last arrival (or, equivalently, shortly before the train's departure). In exchange for the flexibility gained from delayed assignments, the railway company must accept an incoming request~$\request_\timetogo$ if there exists a feasible solution assigning~$\request_\timetogo$ and all requests in~$\requests_{\timetogo-1}$.

From a computational standpoint, \Problem-SFCFS reduces to solving a feasibility problem for each incoming request. More precisely, given a request~$\request_\timetogo$ and a set~$\overline{\requests}_{\timetogo-1}$ of accepted requests, we must decide whether there exists a feasible assignment of all requests in~$\overline{\requests}_{\timetogo-1} \cup\{ \request_\timetogo\}$. Algorithm~\ref{alg:crp-sfcfs} presents the implementation of this idea, to which we refer as the~\texttt{SFCFS Policy}; we abuse notation slightly and use~$\Problem(\overline{\requests}_0 \cup \{ \request_\timetogo\})$ to denote the optimal objective value to the associated instance.

% Algorithm~\ref{alg:crp-sfcfs} describes~\texttt{SFCFS Policy}, a procedure that consists of solving~$\Problem(\overline{\requests}_0 \cup \{ \request_\timetogo\})$ upon the arrival of request~$\request_\timetogo$ and updating the set~$\overline{\requests}$ whenever applicable. We abuse notation slightly and use~$\Problem(\overline{\requests}_0 \cup \{ \request_\timetogo\})$ also to denote the optimal objective value to the associated instance.  
%an assignment decision is made.
\begin{algorithm}[H]
\begin{algorithmic}[1]
\footnotesize
\State Set~$\overline{\requests}_0 = \emptyset$
\For{every time step~$\timetogo$}
\State $\overline{\requests}_{\timetogo} = \overline{\requests}_{\timetogo-1}$
\If{$\Problem(\overline{\requests}_0 \cup \{ \request_\timetogo\}) = \sum_{\request \in \overline{\requests}_0 \cup \{ \request_\timetogo\}}\price(\request) $}
\State $\overline{\requests}_{\timetogo} = \overline{\requests}_{\timetogo-1} \cup \{ \request_\timetogo \}$
\EndIf{}
\EndFor{}
\end{algorithmic}
\caption{\texttt{SFCFS Policy}}
\label{alg:crp-sfcfs} 
\end{algorithm}

This verification can be cast as the decision version of~\Problem where the goal is to decide whether all requests can be packed; recall that, in terms of computational complexity, the optimization version of this problem does not admit a~$(d^{1 - \epsilon})$-approximation for any~$\epsilon > 0$, so effective polynomial-time algorithms are unlikely to exist. Nevertheless, we show in our numerical studies that~$\Problem(\overline{\requests})$ can be solved efficiently. 

% Observe that the~\ref{alg:crp-sfcfs} addresses the offline and the online version of~\Problem-SFCFS. The performance of~\ref{alg:crp-sfcfs} is strongly related to our ability to optimize~\ref{model:offline} quickly. Recall that this problem does not admit a~$(d^{1 - \epsilon})$-approximation for any~$\epsilon > 0$, so effective polynomial-time algorithms are unlikely to exist. 

%Moreover, by definition, the objective value of an exact solution to~\ProblemTemporalFairness provides a lower bound for~\Problem. In our numerical experiments, we investigated the value of temporal fairness, especially in comparison with the easier notion of fairness adopted in~\Problem. 

\section{Numerical Study}\label{sec:experiments}

We report the results of our numerical study involving the algorithms and policies presented in this work. We implement all algorithms in Python 3.10.4 and use Gurobi 11.0.2 to solve the mixed-integer linear programming formulations~\citep{gurobi}. The experiments are executed on an Apple  M1 Pro with 32 GB of RAM. 

We generate instances of~\Problem using parameters of the problem describing arrival rates, seating capacities, and ticket prices observed in the  Tokyo-Shin-Osaka line (see Figure~\ref{fig:shinkansen_map}).
%, which consists of five stations (i.e., four legs): Tokyo, Shin-Yokohama, Nagoya, Kyoto, and Shin-Osaka (see also Figure~\ref{fig:shinkansen_map}). 
The seating capacity is presented in Table~\ref{Table with seating capacities} (extracted from \url{https://en.wikipedia.org/wiki/N700S_Series_Shinkansen}); in situations where the the order of the coaches matter, we always follow the order presented in this table (i.e., 65 seats in the first coach, 100 in the second, etc.).
%; a train consists of 16 coaches with an average seating capacity of 83 passengers.
\renewcommand{\arraystretch}{1.2}
\begin{table}[!ht]
\centering
% \scriptsize
% \footnotesize
\begin{tabular}{c|llllllllllllllll}
\textbf{Car Number} & 1 & 2 & 3 & 4 & 5 & 6 & 7 & 8 & 9 & 10 & 11 & 12 & 13 & 14 & 15 & 16 \\
\hline
\textbf{Seating capacity} & 65 & 100 & 85 & 100 & 90 & 100 & 75 & 68 & 64 & 68 & 63 & 100 & 90 & 100 & 80 & 75 \\
% \hline
\end{tabular}
\caption{Seating capacity of the N700S series Shinkansen train, used in the Nomozi trains. }
\label{Table with seating capacities}
\end{table}
Prices and arrival rates are presented in Table~\ref{Table with prices} (extracted from~\url{https://global.jr-central.co.jp/en/info/fare/_pdf/nozomi.pdf}) and~\citet{manchiraju2022dynamic}. 
\begin{table}[!ht]
\centering
% \footnotesize
\begin{tabular}{*{9}{c|}} %{l|c|c|c|c|c}
    % & \multicolumn{5}{c}{Departure station} \\
    % \cline{2-2}
    & \multicolumn{2}{c}{\textbf{Tokyo}}                        \\
    \cline{1-3}
%\textbf{Shinagawa}     & 2680  & \multicolumn{2}{c}{\textbf{Shinagawa}}        \\
%\cline{1-3}
\textbf{Shin-Yokohama} & JP\textyen3,010 & 87        & \multicolumn{2}{c}{\textbf{Shin-Yokohama}}        \\
\cline{1-5}
\textbf{Nagoya}        & JP\textyen11,300 & 677      & JP\textyen10,640  & 125          & \multicolumn{2}{c}{\textbf{Nagoya}}        \\
\cline{1-7}
\textbf{Kyoto}         & JP\textyen14,170 & 390     & JP\textyen13,500   & 110         & JP\textyen5,910 & 77  & \multicolumn{2}{c}{\textbf{Kyoto}} \\
\cline{1-9}
\textbf{Shin-Osaka}    & JP\textyen14,720 & 846     & JP\textyen14,390   & 175         & JP\textyen6,680 & 232  & JP\textyen3,080 & 61\\
\hline
\end{tabular}
\caption{List of prices and average demand for reserved tickets in the Tokaido Shinkansen line for Nozomi trains; columns represent the departure locations, and rows indicate the destinations. }
\label{Table with prices}
\end{table}
We generate 50 instances and use a sales horizon of 30 days.

\subsection{The Offline~\Problem-FCFS}\label{sec: The Offline CRP-FCFS}

%We start with the investigation of the offline~\Problem-FCFS, which is the version of the problem where all the requests are known and the FCFS fairness condition must be observed, i.e., request~$\request_{\timetogo}$ can only be rejected if the residual capacity in each coach is insufficient given the assignments involving all requests in~$\requests_{\timetogo-1}$.  
% \subsubsection{Algorithms}
We study the performance of three algorithms for the offline~\Problem-FCFS, the version of the problem where all the requests are known and the FCFS fairness condition must be observed, i.e., request~$\request_{\timetogo}$ can only be rejected if the residual capacity in each coach is insufficient given the assignments of all requests in~$\requests_{\timetogo-1}$: %\texttt{FirstFit}, \texttt{RandomFit}, 
%\texttt{Baseline}, 
%and~\texttt{Branch\&Cut}. %Implementation details are discussed below.

\begin{itemize}
    \item {\texttt{FirstFit}:} Each request is assigned to the first coach where it fits, following the order in Table~\ref{Table with seating capacities}.
    
    \item {\texttt{RandomFit}:} Each request~$\request$ is assigned to a randomly chosen coach with sufficient capacity.
    
    \item {\texttt{Branch\&Cut}:} Preliminary experiments showed that Gurobi could not identify a feasible solution to~\ref{model:deterministic} even after 30 minutes. Therefore, we designed \texttt{Branch\&Cut}, 
    which solves the offline~\Problem-FCFS using~\ref{model:deterministic} as base but
    %and the other ideas presented in~\S\ref{sec:The offline CRP-FCFS}.The main aspect of \texttt{Branch\&Cut} is the incorporation of 
    incorporates the fairness-enforcing constraints~\eqref{eq:Unfair Rejections} in real-time as separation cuts instead of adding them in the master problem. Moreover, to enhance computational performance, we also add to the master problem forward filtering inequalities (see~\S\ref{sec: forward filtering}). Preliminary experiments suggested that inequalities~\eqref{eq:Unfair Rejections B} and~\eqref{eq:Unfair Assignments} are computationally expensive, so we use the following family of inequalities to enforce precedence across requests of the same type, which are special cases of~\eqref{eq:Unfair Assignments}:
\begin{eqnarray*}
y_{\requesttype,k} \geq y_{\requesttype,k+1}, \forall \requesttype \in \requesttypes, k \in \mathbb{N} 
\end{eqnarray*}
The fairness-enforcing constraints are separated in polynomial time upon the identification of feasible (integer) solutions. Namely, given a feasible solution~$(\boldsymbol{x},\boldsymbol{y},\boldsymbol{z}$), \texttt{Branch\&Cut} identifies the earliest request~$\request_\timetogo$ (i.e.,  with smallest~$\timetogo$) that has been unfairly rejected. Given~$\request_\timetogo$ and the first coach~$\coach$ (as defined in Table~\ref{Table with seating capacities}) to which it could have been assigned, \texttt{Branch\&Cut} incorporates all inequalities of type~\eqref{eq:Unfair Rejections} associated with~$\request_\timetogo$ and~$\coach$ to separate~$(\boldsymbol{x},\boldsymbol{y},\boldsymbol{z})$. There are~$O(|\requests||\setLegs||\coaches|)$ constraints of type~\eqref{eq:Unfair Rejections}, so the reduction in size of the master problem compensates for the computational overhead of identifying separation cuts in real-time.  Moreover, after adding the cuts for~$\request_{\timetogo}$, we use \texttt{RandomFit} to extend the assignment plan in~$(\boldsymbol{x},\boldsymbol{y}, \boldsymbol{z})$ involving the first~$\timetogo-1$ arrivals (i.e., the subset of requests that were fairly assigned) by assigning (or rejecting) requests~$\request_{k}, \request_{k+1},\ldots,\request_{\numrequests}$, thus obtaining a feasible solution to~\Problem-FCFS.

We also use pre-processing techniques to enhance the computational performance of~\texttt{Branch\&Cut}. One consists of using \texttt{RandomFit} to identify a warm-start solution to the problem. The other technique consists of setting~$y_{\request_\timetogo}$ to 1, thus forcing the assignment of~$\request_\timetogo$,  for each request~$\request_\timetogo$ such that
$$\sum_{\request \in \requests_{\timetogo-1}: \leg \in \setLegs(\request)}\numPass(\request) \leq (1-\delta)|\coaches|\reservedcapacity, \forall \leg \in \setLegs(\request_\timetogo).$$
Remark~\ref{remark:blocking} shows that these assignments do not eliminate feasible solutions, as a train must have sufficient residual capacity to accommodate these requests regardless of the assignment decisions.  
\medskip

\end{itemize}

% \paragraph{Numerical Results}

Our experiments show that~\texttt{Branch\&Cut} is computationally challenging but allows for the identification of high-quality solutions. Namely, when restricted to a time limit of 30 minutes, \texttt{Branch\&Cut}  attains average optimality gaps of 0.23\% and worst-case gap of 0.86\%. 
%In practice, the typical revenue collected from a trip is approximately \$150,000, so this difference translates to approximately \$900. 
The average resource utilization is 93.56\%, with a standard deviation 0.011. These utilization levels are expected, as except during peak seasons, Shinkansen trains are typically sold out for certain itineraries, but not all.
%the seating capacity is not entirely exhausted.

The plots in Figure~\ref{fig:CRPFCFS_offline} show the relative revenue level attained by \texttt{FirstFit} and \texttt{RandomFit} in comparison with the revenue of the best primal solution identified by~\texttt{Branch\&Cut}. Figure~\ref{fig:CRPFCFS_profit} is a cumulative performance plot where each point~$(x,y)$ indicates that the respective algorithm attained at least~$y$\% of the revenue collected by~\texttt{Branch\&Cut} for at least $x$\% of the instances. 
Figure~\ref{fig:CRPFCFS_profit_BA} presents a Bland-Altman plot reporting the differences in performance (also measured in terms of the relative revenue) of \texttt{RandomFit} and \texttt{FirstFit}. Each point represents an instance, with the coordinates indicating the mean relative revenue attained by the algorithms (in the $x$-axis) and the difference between their outcomes ($y$-axis). The plot also presents the limits of agreement with lines delimiting the points deviating by 1.96 standard deviations from the mean.
\begin{figure}[h!]
\centering
\subfloat[\centering  Relative revenue of \texttt{RandomFit} and \texttt{FirstFit}  \label{fig:CRPFCFS_profit}]{%
  \includegraphics[scale=0.495]{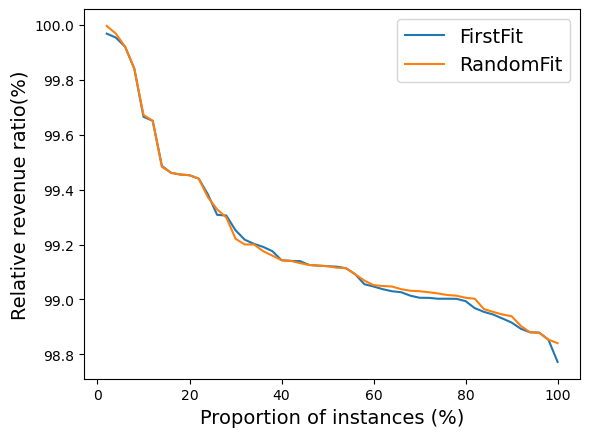}%
}
\subfloat[\centering Comparison between \texttt{RandomFit} and \texttt{FirstFit} \label{fig:CRPFCFS_profit_BA}]{%
  \includegraphics[scale=0.4]{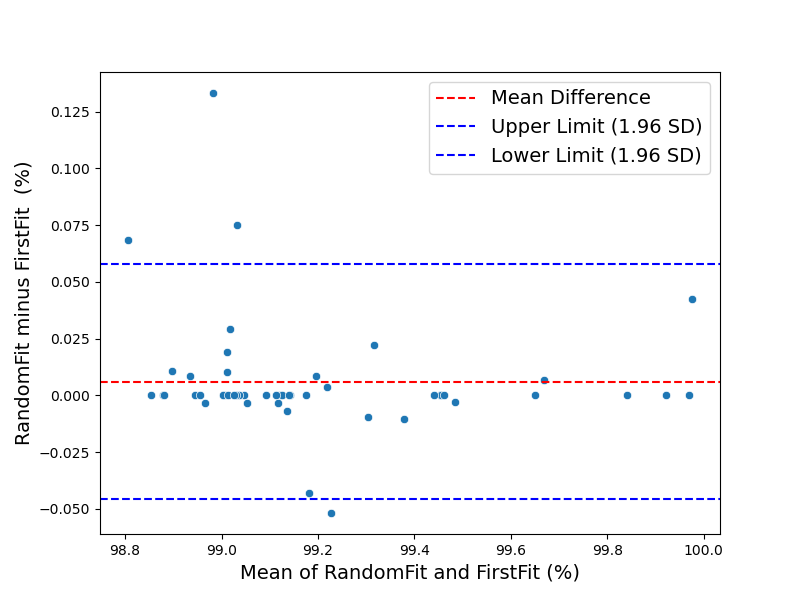}%
}
\caption{Performance profiles of  heuristic policies for \Problem-FCFS.}
\label{fig:CRPFCFS_offline}
\end{figure}

The results show that both heuristics are effective in practice, with losses inferior to 1\% in comparison with \texttt{Branch\& Cut}. Moreover, \texttt{RandomFit} slight outperforms \texttt{FirstFit} (see Figure~\ref{fig:CRPFCFS_profit_BA}). From a practicality standpoint,  \texttt{RandomFit} presents the best trade-off in terms of implementation complexity (the algorithm is fully combinatorial and does not rely on sophisticated mathematical programming techniques), computational complexity (its execution time is negligible), and performance. Finally, \texttt{RandomFit} is inherently online and adherent to the FCFS fairness condition in that it makes irrevocable reject-or-assign decisions sequentially and only rejects requests when it does not have sufficient residual capacity. Therefore, the price of FCFS fairness on the revenue of railway services is negligible. %We also investigate the performance of~\texttt{RandomFit} when applied to online versions of~\Problem. 

\subsection{The Online~\Problem}

%We proceed with the investigation of online~\Problem, which does not consider fairness conditions. 
%\subsubsection{Algorithms}\label{section: CRP algorithms} 
We evaluate the empirical performance of 
four algorithms %the following algorithms 
%studied in Section~\ref{section:Online Policies for CRP} 
for the online~\Problem, which does not consider fairness conditions.
%: \texttt{Fixed}, \texttt{ROM}, and \texttt{Fluid}. 
We also consider \texttt{RandomFit}, as it is also suitable for the problem, and  % \texttt{AdaptiveROM}, 
an adaptation of the policy derived from our study of the random order model 
%~\texttt{ROM} 
that considers additional structural information about the instances used in our study. We use as baseline \texttt{Exact}, which refers to the optimal offline solution of the instances.

\begin{itemize}
    % \item {\texttt{Exact}:} We solve the offline versions of the instances to obtain baseline solutions.
    
    \item {\texttt{Fixed}:} Our implementation of~\texttt{Fixed} solves~\ref{model:primal} before the first time step, using the arrival rates to define the probabilities, and adopts the optimal a priori assignment plan for the entire selling horizon. 
    %More precisely, \texttt{Fixed} assigns request~$\request$ if and only if request~$\request$ is assigned in the a priori plan. 
    We use~$\delta = 0.001$, i.e., \ref{model:primal} considers only requests arriving with probability superior to 0.1\%. 

    \item  {\texttt{ROM}:} For our implementation of~$\texttt{NR}(q)$, we adopt the optimal parameterization presented in Theorem~\ref{thm:randomorder}, with $q = \frac{1}{2-\delta} = 0.515$. Algorithm~$\texttt{NR}(q)$ requires  knowledge about the number~$\numrequests$ of arrivals, which is unknown in practice. Therefore, we use the expected number of requests when defining the length of the sampling phase, i.e., \texttt{ROM} considers the first~$q\sum\limits_{\requesttype \in \requesttypes}\arrivalrate_{\requesttype}$ arrivals as part of the sampling phase. The knapsack phase contains all the other arrivals, so its length is instance-dependent.

  \item {\texttt{AdaptiveROM}:} 
  %The performance of~\texttt{ROM} is relatively poor because it does not consider information about request types and arrival rates. In particular, 
  \texttt{ROM} is oblivious to the fact that several requests may be of the same type, and the implications in our setting are significant because some assignments may become unnecessarily blocked. For example, we may observe the following undesirable combination of events when \texttt{ROM} processes some request~$\request_{\timetogo} = (\requesttype,j)$ during the knapsack phase: a) $\request_{\timetogo}$ is not assigned in~$\Problem(\requests_\timetogo)_{LR}$ (i.e., $\sum_{\coach \in \coaches}x^{(\timetogo)}_{\requesttype, j,\coach} = 0$), and, consequently, gets rejected by~\texttt{ROM}); b) all requests~$(\requesttype,j')$ such that~$j' < j$ (i.e., requests of same type which arrived before~$\request_{\timetogo}$) are assigned by~$\Problem(\requests_\timetogo)_{LR}$; and c) \texttt{ROM} did not assign any requests of type~$\requesttype$ after observing the first~$\timetogo$ request. With that, we have a situation where \texttt{ROM} cannot assign a request for technicalities. This shortcoming of~\texttt{ROM} is mitigated in~\texttt{AdaptiveROM}, an adaptation of~\texttt{ROM} that explores the fact that many requests of the same type may arrive to ``recover'' lost sales opportunities. Namely, after solving~$\Problem(\requests_\timetogo)_{LR}$, we compare the random value~$z$ (see line 5 of Algorithm~\ref{alg:NR}) against~$\left(\sum_{k \in \mathbb{N}}y_{\requesttype,k}\right) - j'$ when deciding on the assignment of~$\request_{\timetogo} = (\requesttype,j)$, where~$j'$ is the number of requests of type~$\requesttype$ assigned by~$\texttt{NR}(q)$ within the first~$\timetogo-1$ time steps. 

\item {\texttt{Fluid}:} A request of type~$\requesttype$ arriving at time step~$\timetogo$ is assigned with probability $\frac{x_{\requesttype,\timetogo,\coach}}{\sum_{\coach \in \coaches_{\reservedcapacity_\timetogo,\request_\timetogo}}x_{\requesttype,\timetogo,\coach}}$ to the first coach where it fits, where~\textbf{x} is the solution to~\texttt{Fluid} with~$\lambda = 0.001$. 
\medskip
\end{itemize}

% \subsubsection{Numerical Results} 

Figure~\ref{fig:CRP_results} reports the relative revenue and utilization of \texttt{Fixed}, \texttt{ROM}, \texttt{AdaptiveROM}, \texttt{Fluid}, and \texttt{RandomFit} in comparison with \texttt{Exact}.
Figure~\ref{fig:CRP_profit} reports the relative revenue
%accumulated profit (or the empirical performance ratio) attained by \texttt{Fixed}, \texttt{ROM}, \texttt{AdaptiveROM}, and \texttt{Fluid} 
(in the $y$-axis) over the selling horizon (in the $x$-axis). We aggregate the results by policy, with each curve indicating the average values and the region within a standard deviation. 
%The results are normalized with respect to the profit attained by~\texttt{Exact}, i.e., each point~$(x,y)$ indicates that, on the~$x$-th day, the respective policy has attained $y$\% of the profit realized by~\texttt{Exact}. 
Note that points~$(x,y)$ where $y > 100$  indicate situations where a policy attains higher revenue than \texttt{Exact}  within the first~$x$ days. The plot on the right (\ref{fig:CRP_utilization}) reports the percentage 
 of the occupied (seating) capacity (in the~$y$-axis) over the selling horizon (in the~$x$-axis).
\begin{figure}[ht!]
\centering
\subfloat[\centering  Sales over time \label{fig:CRP_profit}]{%
  \includegraphics[scale=0.475]{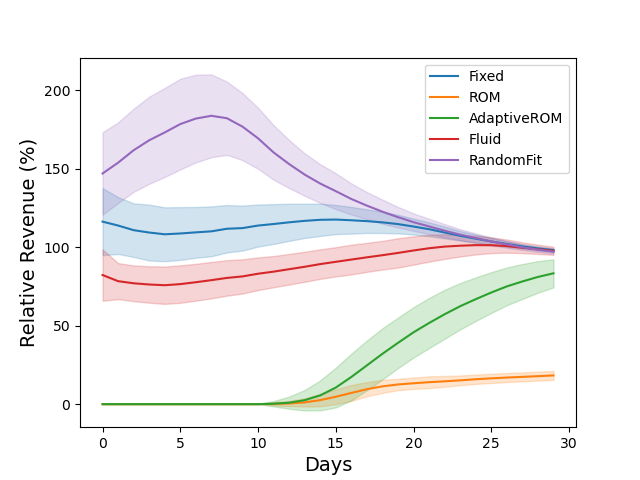}%
}
\subfloat[\centering Utilization over time \label{fig:CRP_utilization}]{%
  \includegraphics[scale=0.475]{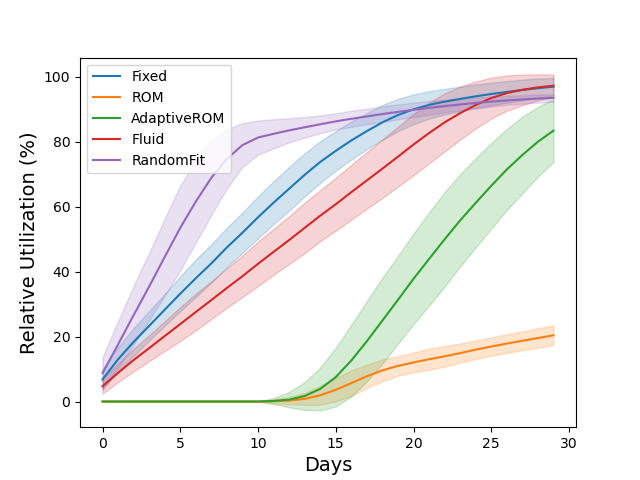}%
}
\caption{Performance profiles for \Problem policies over the selling horizon.}
\label{fig:CRP_results}
\end{figure}

The results show that~\texttt{Fixed}, \texttt{Fluid}, and \texttt{RandomFit} have solid performance and consistently attain quasi-optimal solutions by the end of the planning horizon. On average, the overall performance of~\texttt{Fixed} is slightly better, with an average empirical ratio of 98.15\%, standard deviation of 1.18, and worst-case ratio of 92.32\%. \texttt{Fluid} attains 97.67\%, 2.57, and 82.17\% for the same indicators. Interestingly, \texttt{RandomFit} is slightly inferior but more stable, with an average of 
97.13\%, a standard deviation of 1.26, and a worst-case ratio of 94.21\%. \texttt{RandomFit} has a clear advantage in the 
%Nevertheless, there are clear differences between the cumulative performance of these two policies over time. In particular, \texttt{Fixed} has better sales results in the 
first half of the selling horizon and consistently outperforms \texttt{Exact} in the period, whereas \texttt{Fixed}'s revenue curve is slightly above that of \texttt{Exact} and \texttt{Fluid} is slightly below. These observations are important in practice and motivate the adoption of \texttt{RandomFit} and \texttt{Fixed} because early sales are reflected in important aspects of railway management operations, such as increased cash flow and lower marketing expenditure. 

The results also show that the sampling phase significantly impacts the policies derived from our study of the online~\Problem in the random order model.  Recall that we use~$q = 0.515$, so more than 50\% of the requests are rejected in expectation, as shown in Figure~\ref{fig:CRP_results}.
%; in particular, Figure~\ref{fig:CRP_results}shows that~\texttt{ROM} and~\texttt{Adaptive-ROM} do not assign requests before the $10^{\text{th}}$ selling day. 
The results suggest that the arrival rates and the selling horizon make the remaining selling days insufficient for both algorithms to attain the performance of~\texttt{Fixed} and~\texttt{Fluid}. \texttt{ROM} is worse than all other policies, with an average empirical ratio of only 18.29\%. This result is slightly worse than the asymptotic rate of 29\% derived from our theoretical analysis, and the difference can be explained by relatively small number of requests in the knapsack phase. For example, tests involving similar instances but with a selling horizon of 60 days attain an average empirical rate of 25.37\%. In contrast, \texttt{Adaptive-ROM} has much better performance, with an average empirical performance of 83.32\%. We note that~\texttt{Adaptive-ROM} has a similar performance to \texttt{Fixed} once the knapsack phase starts in that both algorithms assign approximately 85\% of the seating capacity within 20 days. Interestingly, a 60-day selling horizon increases the average empirical performance of~\texttt{Adaptive-ROM}  to 84.13\%, a minor difference with respect to the results for 30-day selling horizons. 

%Finally, the results show that \texttt{RandomFit} has remarkable performance also when applied to the online~\Problem. As expected, the results are much stronger in the beginning of the selling horizon, as~\texttt{RandomFit} rejects only requests that it cannot assign. \texttt{Fluid} and \texttt{Fixed} are still slightly better when we consider the entire selling horizon; in particular, we see that \textbf{XXX}. As \texttt{Fluid} and \texttt{Fixed} are relatively simple algorithms that require modest real-time computational power, the case for the practical adoption of those policies is stronger. Nevertheless, \texttt{RandomFit} is arguably a satisfactory approach in scenarios where the decision maker has limited knowledge about the requests.

\subsection{Online~\Problem with Fairness Constraints}

We conclude by evaluation algorithms tailored for the online \Problem-FCFS and \Problem-SFCFS. In particular, we investigate the impact of the fairness constraints on the revenue attained by the algorithms. We measure the impact of fairness by using the results of~\texttt{Exact} as baseline. 
%\subsubsection{Algorithms} We study two algorithms tailored for online versions of~\Problem  with fairness constraints. 
We also include \texttt{RandomFit} in the analysis, as it can be used to solve the online~\Problem-FCFS.

\begin{itemize}
    \item {\texttt{SFCFS}:} We use~\texttt{SFCFS} to denote our implementation of the algorithm that solves \Problem-SFCFS. 

    \item {\texttt{PreAssign}:} Similar to \texttt{Fixed}, \texttt{PreAssign} solves~\ref{model:primal} with~$\delta = 0.001$ based on the arrival rates to obtain an a priori assignment plan. 
    %If an incoming request~$\request_\timetogo$ has not been pre-assigned but can be serviced, \texttt{PreAssign} solves~\ref{model:primal} again while forcing the assignment of~$\request_\timetogo$. For these updates in the pre-assignment plan, \texttt{PreAssign} does not change the arrival probabilities of the requests are not changed. However, 
    \texttt{PreAssign} updates the probabilities and solves~\ref{model:primal} at the beginning of each day to obtain an assignment plan that reflects current probabilities and residual capacities. Updates to the formulation made during the day to accommodate requests that have not been pre-assigned do not incorporate changes in the probabilities. %The idea is that \texttt{PreAssign} will make significant updates once per day and update the plan during the day whenever necessary (i.e., when an incoming request has not been pre-assigned).
\medskip
\end{itemize}

%\subsubsection{Numerical Results} 
%In terms of computational performance, our experiments suggest that the 
Our experiments shows that the computational burden of finding a feasible solution to~$\Problem(\overline{\requests}_0 \cup \{ \request_\timetogo\})$ upon each arrival~$\request_\timetogo$ is negligible. Therefore, we report the results of an exact (optimal) implementation of~\texttt{SFCFS}. \texttt{PreAssign} is computationally more challenging, but the average solution time per request is inferior to one second, so the algorithm is also adequate for practical adoption. 

Figure~\ref{fig:Fairness_results} presents Bland-Altman plots to report the relative performance of \texttt{SFCFS}, \texttt{PreAssign}, and \texttt{RandomFit}; we use again~\texttt{Exact} as reference. The comparison between \texttt{PreAssign} and  \texttt{RandomFit} in Figure~\ref{fig:onlineFCFS_profit} 
%reports statistics on the difference between the revenue attained by \texttt{PreAssign} and  \texttt{RandomFit} for each instance. 
shows that \texttt{PreAssign} delivers better results for more instances, but some outliers make the average revenue attained by \texttt{RandomFit} marginally larger (by less than 0.01\%). Figure~\ref{fig:CRPFCFS_utilization} reveals that \texttt{SFCFS} and  \texttt{RandomFit} are similar, with outlier instances given some aggregate advantage for \texttt{RandomFit}. 

%Each plot compares two algorithms by indicating, for each instance, the mean revenue attained by the algorithms (in the $x$-axis) and the difference between their outcomes ($y$-axis). The plots also indicate the limits of agreement by indicating the points deviating by 1.96 standard deviations from the mean.
\begin{figure}[h!]
\centering
\subfloat[\centering  Comparison between \texttt{PreAssign} and \texttt{RandomFit} \label{fig:onlineFCFS_profit}]{%
  \includegraphics[scale=0.4025]{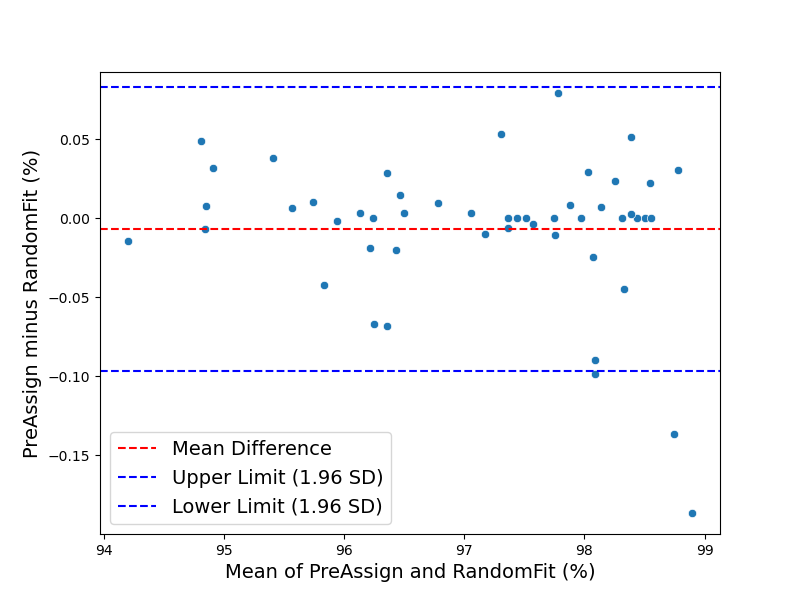}%
}
\subfloat[\centering Comparison between \texttt{SFCFS} and \texttt{RandomFit} \label{fig:CRPFCFS_utilization}]{%
  \includegraphics[scale=0.4025]{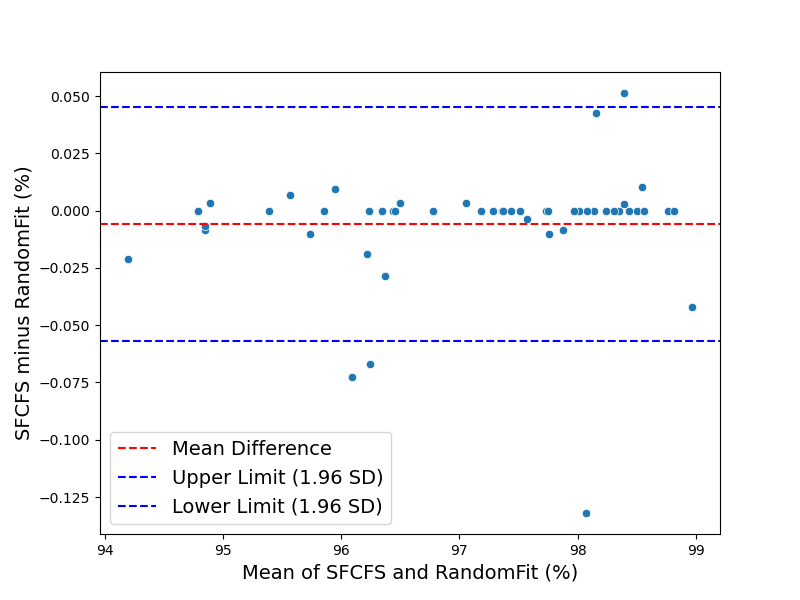}%
}
\caption{Performance of policies that observe fairness considerations.}
\label{fig:Fairness_results}
\end{figure}

Our results provide a strong case in favor of SFCFS in practical settings. The differences in revenue between \texttt{SFCFS} and \texttt{RandomFit} are inferior to~$0.1\%$, so even price adjustments may be unnecessary. For reference, 0.1\% corresponds to approximately \$150 (or JP\textyen20,000), or two train tickets on the most expensive routes (see Table~\ref{Table with prices}), so increased brand perception and customer goodwill derived from adopting SFCFS should provide a satisfactory trade-off for the railway company. Moreover, the results shows that the stricter notion of fairness is adequately balanced by higher flexibility in the assignment decisions.

\section{Conclusions and Future Work}

This work presents the first study of the online seat reservation problem for railways for group requests, which encompasses online decisions on decisions and assignment of requests to coaches. We study different settings of the problem, motivated by similar problems investigated in the literature and by real-world applications modeled by~\Problem. The offline version of~\Problem-FCFS is computationally challenging, so we design a mathematical programming algorithm to solve it exactly. Moreover, we study different online models of~\Problem and obtain algorithms with constant-factor competitive ratios, including one that extends to general packing problems that have been investigated in the literature. Our numerical experiments involving instances based on real-world scenarios show that the policies based on our algorithm are adequate in practice. Moreover, our results provide a strong support in favor of the adoption of fairness constraints in railway systems, as the impact in the revenue is minor.

There are interesting directions for future work.  From the operational standpoint, we note that, in practice, other assignment constraints may be associated with group requests. For example, in addition to enforcing the assignment of all passengers to the same coach, the railway company may also wish to restrict the distribution of passengers to seats to a subset of templates, thus ensuring that group members are seated nearby. These conditions require seat-level assignments for groups, which are analytically and computationally more challenging to study.

% Acknowledgments here
% \ACKNOWLEDGMENT{The authors gratefully thank the reviewers of POM.}

%%REFERENCES%%
%%%%%%%%%%%%%%%%%%%%%%%%%%%%%%%%%%%%%%%%%%%%%%%%%%%%%%%%%%%%%%%%%%%%%%%%%%%%%%%%%%%%%%%%%%%%%%%%%%%%%%%%%%%%%%%%%%%%%%%%%%%%%%%%%%%%
%% This template complies references using bibtex. You will need to use pomsref.bst file for biblography style.
%REFERENCES USING BIBTEX FILES
%%%%%%%%%%%%%%%%%%%%%%%%%%%%%%%%%%%%%%%%%%%%%%%%%%%%%%%%%%%%%%%%%%%%%%%%%%%%%%%%%%%%%%%%%%%%%%%%%%%%%%%%%%%%%%%%%%%%%%%%%%%%%%%%%%%%

\bibliographystyle{plainnat}

\bibliography{refs}
%%%%%%%%%%%%%%%%%%%%%%%%%%%%%%%%%%%%%%%%%%%%%%%%%%%%%%%%%%%%%%%%%%%%%%%%%%%%%%%%%%%%%%%%%%%%%%%%%%%%%%%%%%%%%%%%%%%%%%%%%%%%%%%%%%%%

%Hayes, R. H., G. P. Pisano. 1996. Manufacturing strategy: At the intersection of two paradigm shifts. Production and Operations Management, 5 (1), 25-41.

%%%%%%%%%%%%%%%%%%%%%%%%%%%%%%%%%%%%%%%%%%%%%%%%%%%%%%%%%%%%%%%%%%%%%%%%%%%%%%%%%%%%%%%%%%%%%%%%%%%%%%%%%%%%%%%%%%%%%%%%%%%%%%%%%%%%
%% %If you don't use BiBTex, you can manually itemize references as shown in the referneces for the electronic comapanion. See below.
 %%%%%%%%%%%%%%%%%%%%%%%%%%%%%%%%%%%%%%%%%%%%%%%%%%%%%%%%%%%%%%%%%%%%%%%%%%%%%%%%%%%%%%%%%%%%%%%%%%%%%%%%%%%%%%%%%%%%%%%%%%%%%%%%%%%%

%% Here starts the e-companion (EC). Place your appendix content here.
%%%%%%%%%%%%%%%%%%%%%%%%%%%%%%%%%%%%%%%%%%%%%%%%%%%%%%%%%%
% \ECSwitch % Comment this line out if you do not need e-companion.
%%%%%%%%%%%%%%%%%%%%%%%%%%%%%%%%%%%%%%%%%%%%%%%%%%%%%%%%%%

%%% Main head for the e-companion
% \ECHead{E-Companion for POM Journal Template}

\appendix

\section{Proof of Corollary~\ref{cor: fluid optimality group 1}}

We generalize the $\arrivalrate$-Policy by making resources for requests of type~$\requesttype$ available   in time step~$\timetogo$ with probability~$\theta \sum\limits_{\coach \in \coaches} \frac{x_{\requesttype,\timetogo,\coach}}{\arrivalrate_{\requesttype}}$, where~$0 < \theta \leq 1$ is a parameter of the policy. The expected revenue~$E[\profit]$ is written as
%in this case is
%attained by this version of the policy is given by
\begin{eqnarray*}
E[\profit] = 
        \sum_{\timetogo \in \timehorizon}
        \sum_{\requesttype \in \requesttypes}
        \price(\requesttype)
        \underbrace{Pr[G_{\requesttype,\timetogo} = 1]}_{\text{Sufficient residual capacity to assign~$\requesttype$}}
        \underbrace{
        Pr[X \geq \timetogo]
        }_{\text{There are at least $\timetogo$ requests.}}
        \underbrace{
        \sum_{\coach \in \coaches}
             \theta x_{\requesttype,\timetogo,\coach}.
            }_{\text{Policy tries to assign a request of type~$\requesttype$.}}      
\end{eqnarray*}
Using the same Bernoulli variables, we obtain
\begin{eqnarray*}
Var\left[N_{\leg,\timetogo}\right]
\leq
    E\left[N_{\leg,\timetogo}\right]
=
    Pr[N_{\leg,\timetogo}] 
&=& 
    \sum_{\requesttype \in \requesttypes: \leg \in \setLegs(\requesttype)}
    \sum_{\coach \in \coaches}
        \theta \frac{x_{\requesttype,\timetogo,\coach}}{\arrivalrate_{\requesttype}}
        \arrivalrate_{\requesttype}    
       =
    \sum_{\requesttype \in \requesttypes: \leg \in \setLegs(\requesttype)}
    \sum_{\coach \in \coaches}
    \theta x_{\requesttype,\timetogo,\coach},
\end{eqnarray*}
and the stochastic features of~$\sum\limits_{\timetogo \in \timehorizon} N_{\leg,\timetogo}$ are such that
\[
\sum_{\timetogo \in \timehorizon}Var\left[N_{\leg,\timetogo}\right]
\leq
    \sum_{\timetogo \in \timehorizon}
        E\left[N_{\leg,\timetogo}\right]
=
    \sum_{\timetogo \in \timehorizon}
    \sum_{\requesttype \in \requesttypes: \leg \in \setLegs(\requesttype)}
    \sum_{\coach \in \coaches}
     \theta x_{\requesttype,\timetogo,\coach}
\leq
    \sum_{\timetogo \in \timehorizon}
    \sum_{\requesttype \in \requesttypes: \leg \in \setLegs(\requesttype)}
    \sum_{\coach \in \coaches}
        q_{\requesttype,\leg} \theta x_{\requesttype,\timetogo,\coach}
\leq
    \theta |\coaches|\reservedcapacity.
\]
The probability which  the number of assigned requests surpasses~$\frac{|\coaches|}{\delta}$ is bounded by:
\begin{eqnarray*}
Pr\left[
    \sum_{\timetogo \in \timehorizon}
        N_{\leg,\timetogo} 
\geq
    \frac{|\coaches|}{\delta}
\right]
&\leq&
Pr\left[
    \sum_{\timetogo \in \timehorizon}
        (N_{\leg,\timetogo} - E[N_{\leg,\timetogo}] ) 
\geq 
    |\coaches|\left(
    \frac{1}{\delta} 
    -\theta\reservedcapacity\right)
\right]
\\
&\leq&
\exp{\left(  
    -\frac{  
        \frac{|\coaches|^2}{2} 
            \left(
                \frac{1}{\delta} 
    -\theta\reservedcapacity
            \right)^2 
        }    
        {
            \sum\limits_{\timetogo \in \timehorizon}
                Var(N_{\leg,\timetogo}) 
            + 
            \frac{|\coaches|}{3} 
                \left(
                   \frac{1}{\delta} 
    -\theta\reservedcapacity
                \right) 
            }  
\right)}
\\
&\leq&
\exp{\left(  
    -\frac{  
        \frac{|\coaches|^2}{2} 
            \left(
                \frac{1}{\delta} 
    -\theta\reservedcapacity
            \right)^2 
        }    
        {
            \theta|\coaches|\reservedcapacity
            + 
            \frac{|\coaches|}{3} 
                \left(
                   \frac{1}{\delta} 
    -\theta\reservedcapacity
                \right) 
            }  
\right)}
=
% \\
% &=&
\exp{\left(  
    -\frac{  
        \frac{|\coaches|}{2} 
            \left(
                \frac{1}{\delta} 
    -\theta\reservedcapacity
            \right)^2 
        }    
        {
            \theta\reservedcapacity
            + 
            \frac{1}{3} 
                \left(
                   \frac{1}{\delta} 
    -\theta\reservedcapacity
                \right) 
            }  
\right).}
\end{eqnarray*}
With that, we obtain the following guarantee for~$E[\profit]$:
\begin{eqnarray*}
    E[\profit]   
&=& 
    \theta
    Pr[G_{\requesttype,\timetogo} = 1]
    \sum_{\timetogo \in \timehorizon}
    \sum_{\requesttype \in \requesttypes}
    \sum_{\coach \in \coaches}
        Pr[X \geq t]
        \price(\requesttype)     
        x_{\requesttype,\timetogo,\coach}
\geq
    \theta \left[
    1 - 
|\setLegs|
\exp{\left(  
    -\frac{  
        \frac{|\coaches|}{2} 
            \left(
                \frac{1}{\delta} 
    -\theta\reservedcapacity
            \right)^2 
        }    
        {
            \theta\reservedcapacity
            + 
            \frac{1}{3} 
                \left(
                   \frac{1}{\delta} 
    -\theta\reservedcapacity
                \right) 
            }  
\right)}
\right]\overline{\profit},
\end{eqnarray*}

By replacing the real-world parameters~$|\coaches| = 20$, $\reservedcapacity = 100$, and~$\delta = 0.06$, we obtain
\begin{eqnarray*}
    J   
&\geq&
    \theta 
    \left(
    1 
    - 
    4 \exp{\left(  
    -\frac{  
        10 
            \left(
                16.667 - 100\theta 
            \right)^2 
        }    
        {
            66.667\theta
            +
            5.555
        }  
\right)} \right)J^d
\geq
\theta \left[
    1 - 
5\exp{\left(  
    -\frac{  
        \frac{1}{2} 
            \left(
                \frac{20}{0.06} 
    -1
    -6\theta
            \right)^2 
        }    
        {
           6\theta
            + 
            \frac{1}{3} 
                \left(
                   \frac{20}{0.06} 
    -1
    -6\theta
                \right) 
            }  
\right)}
\right],
\end{eqnarray*}
which is monotonically increasing for~$\theta > 0.2$ and, in particular, arbitrarily close to 1 for~$\theta = 1$, so the extended policy converges is identical to the original one for~$\delta = 0.06$. More generally, $\theta = 1$ is optimal for every~$\delta \geq \frac{2}{\reservedcapacity}$, i.e., the policy derived from~\ref{model:fluid}. For~$\delta = \frac{1}{\reservedcapacity}$, the maximum is attained by~$\theta \approx 0.9218$, which yields an approximation guarantee of~$0.916$, whereas~$\theta = 1$ gives a trivial (negative) lower bound.

\section{Formulation for Deterministic \Problem}\label{sec:Deterministic Formulation}

Formulation~\ref{model:deterministic} presented below solves~\Problem.
% \[
\begin{align}\label{model:deterministic}
\tag{\textbf{Offline}}
    \max\  &\sum\limits_{\request \in \requests}
    \price(\requesttype) \cdot  y_{\request} &  \\
&
    \sum\limits_{\coach \in \coaches}  x_{\request,\coach} \leq y_{\request}, 
    % & 
    \forall \request   \in \mathcal{R}  \nonumber \\
%(\beta) 
&
    \sum\limits_{\substack{(\request,\coach) \in \assignments \\ 
    \leg \in \setLegs(\requesttype)}} \numPass(\requesttype) \cdot x_{\request,\coach} \leq \reservedcapacity,
    % & 
    \forall (\leg,\coach) \in  \setLegs \times \coaches  \nonumber \\    
&
    y_{\request} + \sum_{\leg \in \setLegs(\request)} z_{\request,\leg,\coach} = 1,  \forall (\request,\coach) \in \requests \times \coaches \nonumber
    \\
&
        \sum_{\request' \in \requests_{\arrivalorder(\request)}: 
        \leg \in \setLegs(\request'), 
        } 
            \numPass(\request') x_{\request',\coach} 
    \geq  
        (\reservedcapacity -   \numPass(\request)+1)z_{\request,\leg,\coach},
        \forall \request,\leg,\coach \in \requests \times \setLegs \times \coaches \nonumber
\\
%(\gamma) 
& 
    y_{\request}  \in \{0,1\}, 
    z_{\request,\leg,\coach} \in \{0,1\}, 
    x_{\request,\coach} \in \{0,1\}, 
    % &
    \forall (\request,\coach) \in \assignments, \leg \in \setLegs \nonumber
\end{align}
% \]
\ref{model:deterministic} uses binary variable~$x_{\requesttype,j,\coach}$ to represent the assignment of request~$\request = (\requesttype,j)$ to coach~$\coach$,  binary variables~$y_{\request}$ to indicate that request~$\request$ has been assigned,
and binary variables~$z_{\request,\leg,\coach}$ to indicate that the residual capacity of leg~$\leg$ on coach~$\coach$ is strictly smaller than~$\numPass(\request)$ upon the arrival of~$\request$. The objective function gives the sum of the revenues collected from all assignments. The first set of constraints asserts that each request is assigned to at most one coach and links the variables~$\boldsymbol{x}$ and~$\boldsymbol{y}$. The second set of (knapsack) constraints controls the capacity utilization of each coach per leg. 

The last two families of constraints work in tandem to enforce fairness by precluding unfair rejections. The first forces the activation of exactly one variable~$z_{\request,\leg,\coach}$ for each coach~$\coach$ if~$\request$ is rejected; otherwise, all these variables are equal to zero. The second family sets a lower bound of~$\reservedcapacity -   \numPass(\request)+1$ on the utilization of~$\leg$ in coach~$\coach$ whenever~$z_{\request,\leg,\coach}$ is activated.

%\end{APPENDIX}
%
%   or
%
% \begin{APPENDICES}
% \section{<Title of Section A>}
% \section{<Title of Section B>}
% etc
% \end{APPENDICES}

%%%%%%%%%%%%%%%%%
\end{document}